\documentclass{article}[12pt]
\usepackage[a4paper, total={6.3in, 8.3in}]{geometry}
\usepackage{amssymb}
\usepackage{amsthm}
\usepackage{amsmath}
\usepackage{enumitem}
\usepackage{hyperref}
\usepackage{authblk}
\usepackage{mathtools}

\newtheorem{theorem}{\large Theorem}[section]
\newtheorem{lemma}{\large Lemma}[section]

\newtheorem{definition}{\large Definition}[section]
\newtheorem{remark}{\large Remark}[section]
\newtheorem{example}{\large Example}[section]

\newtheorem{proposition}{\large Proposition}[section]

\def\1{\rule{0pt}{1.7ex}xy}

\def\2{\rule{0pt}{1.7ex}x_0x}
\def\3{\rule{0pt}{2ex}X_x}

\def\4{\rule{0pt}{1.7ex}1}
\def\5{\rule{0pt}{1.7ex}2}

\begin{document}

	\thispagestyle{empty}
	
	\title{The Karush-Kuhn-Tucker Optimality Conditions For Multi-Objective Interval-Valued Optimization Problem On Hadamard Manifolds}
	
	\author[a]{Hilal Ahmad Bhat}

	\author[b,*]{Akhlad Iqbal}

     \author[c]{I. Ahmad}

\affil[a,b]{\it \small Department of Mathematics, Aligarh Muslim University, Aligarh, 202002, Uttar Pradesh, India}

\affil[c]{\it \small Department of Mathematics, King Fahd University of Petroleum \& Minerals, Dhahran-31261, Saudi Arabia}
\date{}

\maketitle

\let\thefootnote\relax\footnotetext{*corresponding author (Akhlad Iqbal)\\ Email addresses: bhathilal01@gmail.com (Hilal Ahmad Bhat),\\ akhlad6star@gmail.com (Akhlad Iqbal),\\
drizhar@kfupm.edu.sa (I. Ahmad).}


\begin{abstract}
	The KKT optimality conditions for multi-objective interval-valued optimization problem on Hadamard manifold are studied in this paper. Several concepts of Pareto optimal solutions, considered under LU and CW ordering on the class of all closed intervals in $\mathbb{R}$, are given. The KKT conditions are presented under the notions of convexity, pseudo-convexity and generalized Hukuhara difference. We show, with the help of an example, that the results done in this paper for solving multi-objective interval-valued optimization problems on Hadamard spaces are more general than the existing ones on Euclidean spaces.  The main results are supported by examples.\\
	
	
\end{abstract}

\section{Introduction}
The uncertainty and lack of accuracy in the day to day human life is inevitable and to tackle such problems, many mathematicians are led to focus on different interesting areas of research such as stochastic optimization programming (SOP), fuzzy optimization programming (FOP) and interval-valued optimization programming (IVOP). Stochastic optimization programming problems include the use of random variables of known probability distribution as the coefficients of the functions involved whereas fuzzy optimization programming problems include the fuzzy numbers with known membership functions for the same. The techniques developed to solve the stochastic and fuzzy optimization programming problems are very subjective and it is difficult to find a proper membership function or a probability distribution due to lack of sufficient data. Moreover, it sometimes becomes complex to relate both of these programming methods to real life problems. So, as an alternative, IVOP provides a better and much easier solution to such uncertain optimization problems which uses a closed interval to illustrate the uncertainty of a variable. In addition, the coefficients involved in the functions of an IVOP are taken as the closed intervals in $\mathbb{R}$.

Many mathematicians have so far explored various methods to solve the IVOP problems. An overview on the application of interval arithmetics was given by Alefeld  and Mayer \cite{alefeld'}. Ishibuchi and Tanaka \cite{ishibuchi} introduced the ordering relation between two closed intervals which involves the center and half width of the closed intervals and derived the solution concepts for the multi-objective IVOP problems. Wu \cite{Wu} provided two solution concepts for an IVOP problem by taking into consideration the two types of partial ordering (LU and CW ordering) on the set of all closed intervals in $\mathbb{R}$,  and later, Wu \cite{Wu1} extended these solution concepts for multi-objective IVOP to type-I (with respect to LU ordering) Pareto optimal solution (POS) and type-II POS (with respect to CW ordering) and derived the KKT optimality conditions for the IVOP. Although, some latest work on KKT-type optimality conditions for IVOP can be seen in \cite{FR, LS}

Further, several researchers lay focus on the extension of the methods and techniques of non-linear analysis in Euclidean spaces to Riemannian manifolds (see \cite{bento', bento, ferreira, li, nemeth, rapcsak, rapcsak', gabriel, udriste, wang}). Such extensions especially in optimization programming from Euclidean spaces to Riemannian manifolds has its own benefits such as  certain functions involved in optimization programming fail to be convex in linear space but the same functions turnout to be convex when introduced on Riemannian manifold under a suitable Riemannian metric \cite{rapcsak, rapcsak', chen}. Moreover, a non-monotone vector field  transforms into a monotone vector field when introduced on a Riemannian manifold with a proper Riemannian metric \cite{ferreira, rapcsak, rapcsak'}. Udriste \cite{udriste} and Rapcsak \cite{rapcsak'} are the first authors who considered the concept of generalized convexity known as geodesic convexity. Chen \cite{chen} extended the concepts of convexity and pseudo-convexity of interval-valued functions and derived the KKT optimality conditions for IVOP on Hadamard manifolds.

Motivated and inspired by above work, we lay our attention on the multi-objective optimization programming problems on Hadamard manifolds in which objective functions are interval-valued and the constraint functions are real valued. We derive the KKT optimality conditions in multi-objective optimization program using the notions of convexity, pseudo-convexity and generalized Hukuhara difference. The examples are provided in support of the main results. Moreover, with the help of a numerical example, we show that results presented in this paper for solving MIVOP on Hadamard manifold are more general than the existing ones on Euclidean spaces (see \cite{ycc}, \cite{eh}, \cite{Wu}, \cite{Wu1}, \cite{jz}). The comparison between the proposed KKT conditions and the existing ones is given in the Section 5. In Section 6, we discuss the relationship between fuzzy set theory and interval analysis.

\section{Preliminaries}
Let $\mathbb{R}^n$ be the $n$-dimensional Euclidean space, then, for any two vectors $p = (p_1, p_2,...,p_n)$ and $q = (q_1, q_2,..., q_n)$, we use the following convention for ordering relations:
\begin{enumerate}[label=(\roman*)]
	\item $p = q \Longleftrightarrow p_i =q_i, ~ i \in \{ 1,2,...,n\}$;
	\item $p > q \Longleftrightarrow p_i > q_i, ~ i \in \{ 1,2,...,n\}$;
	\item $p \geqq q \Longleftrightarrow p_i \geqq q_i,~  i \in \{ 1,2,...,n\}$;
	\item $p \geq q \Longleftrightarrow p \geqq q ~ \text{and}~ p \neq q. $
\end{enumerate}

\begin{definition} \rm \cite{udriste}
	A Hadamard manifold $H$ is a complete simply connected Riemannian manifold of non-positive sectional curvature.
\end{definition}	

\begin{proposition}
	{\rm \cite{udriste}} Let $H$ be a Hadamard manifold and $p \in H.$ Then $exp_p: T_p (H) \rightarrow H$ is a diffeomorphism and for any $p,q \in H,$ there exists a unique normalized geodesic $\gamma_{p,q} = exp_p(\alpha~exp^{-1}_p q)$, for all $\alpha \in [0,1]$, joining p to q.
\end{proposition}

\begin{definition}
	{\rm \cite{udriste}} \rm Let $H$ be a Hadamard manifold. A subset $E \subseteq H$ is said to be geodesic convex if, for all $p$, $q$ $\in$ $E$, the geodesic joining $p$ to $q$ is contained in $E$; i.e., if $\gamma : [0,1] \rightarrow H$ is a geodesic such that $p =\gamma(0)$ and $q=\gamma (1)$, then $\gamma_{p,q}(\alpha) = exp_p(\alpha ~ exp_p^{-1}q) \in E$, for all $\alpha \in [0,1]$.
\end{definition}

\begin{definition} \rm \cite{li} Let $H$ be a Hadamard manifold, $E \subseteq H$ be a nonempty open set, and let $\phi:E \rightarrow \mathbb{R}$ be a function. We say that $\phi$ is directionally differentiable at a point $p \in E$ in the direction of $w \in T_p(H)$ if the limit 
	$$\phi'(p;w) = \lim\limits_{\alpha \rightarrow 0^+} \frac{\phi(exp_p ~ \alpha w)-\phi(p)}{\alpha}$$
	exists, where $\phi'(p;w)$ is called the directional derivative of $\phi$ at $p$ in the direction of $w \in T_p(H).$  If $\phi$ is directionally differentiable at $p$ in every direction $w \in T_p(H)$, we say $\phi$ is directionally differentiable at $p$.
\end{definition}

\begin{definition}
	\rm \cite{udriste} Let $H$ be a Hadamard manifold and $E\subseteq H$ be a nonempty open geodesic convex set. A function $\phi: E \rightarrow \mathbb{R}$ is said to be convex at $p \in E$, if for any $q \in E$,
	$$\phi(\gamma(\alpha))\leqq \alpha \phi(q) +(1-\alpha)\phi(p), ~~~ \text{for all} ~ \alpha \in[0,1].$$
	Where $\gamma : [0,1] \rightarrow E$ is a geodesic with $\gamma(0) = p$ and $\gamma(1) = q$. If the above inequality is strict for $p\neq q$, then we say that $\phi$ is a strict convex function at $p$. Moreover, if the above inequality holds, for all $p,q \in E$, then we say that $\phi$ is a convex function in $E$.  
\end{definition}

\begin{remark}\label{remark2.1,p1}
	\rm (i) if $\phi$ is differentiable at $p$, then, 
	$$\phi'(p,w) = ~<\nabla \phi(x), ~w>,$$
	where $\nabla \phi(p)$ denotes the gradient of $\phi$ at $p$ and $w \in T_p(H)$.\\
	\\
	(ii)  A function $\phi$ is a convex function on $E$ if and only if, for all $p,q \in E$, $\phi(q) -\phi(p) \geqq \phi'(p, exp^{-1}_p q).$
\end{remark}

\noindent Now, we recall the arithmetics of intervals:

Let $\mathbb{I}$ be the set of all closed and bounded intervals in $\mathbb{R}$, let $T \in \mathbb{I}$, we write $T=[t^L,t^U]$, where $t^L$ and $t^U$ are respective lower and upper bounds of $T$. Let $T = [t^L, t^U],~ S= [s^L, s^U] \in \mathbb{I}$ and $k \in \mathbb{R}$, we have  
\begin{align*}	
	T + S &= \{t+s: t\in T \text{~and ~} s \in S\} = [t^L+s^L, t^U+s^U],\\
	kT &= \{kt: t \in T\} = \left \{ \begin{array}{cc}
		[kt^L,kt^U], & k \geqq 0,\\
		|k|[-t^U, -t^L], & k<0.
	\end{array} \right.
\end{align*}
Therefore, $-T = \{-t: t \in T\} = [-t^U,-t^L],$ and $T-S = T+(-S) = [t^L-s^U, t^U-s^L].$ The real number $\alpha \in \mathbb{R}$ can be regarded as a closed interval $T_{\alpha} = [\alpha,\alpha]$. The Hausdorff metric between intervals $T$ and $S$  is defined as $d_H(T,S) = \text{max}\{|t^L-s^L|, |t^U-s^U|\}.$ For more details on this, we refer to Wu \cite{Wu}, Alefeld and Herzberger \cite{alefeld} and Moore \cite{moorem, moorei}.

Let $H$ be a Hadamard manifold. The function $\phi:H\rightarrow \mathbb{I}$ is called an interval-valued function, and denoted as $\phi(p) = [\phi^L(p), \phi^U(p)],$ where $\phi^L$ and $\phi^U$ are real valued functions and satisfy $\phi^L(p) \leqq \phi^U(p)$ for every $p \in H$.  \\

{\bf \noindent Solution Concepts:}\\
\noindent For ordering relation between intervals in $\mathbb{I}$, we use the following convention:\\
Let $T=[t^L,t^U]$, $S = [s^L,s^U]$ $\in$ $\mathbb{I}$, we write
\begin{enumerate}[label=(\roman*)]
	\item $T \leqq_{LU} S \Longleftrightarrow t^L \leqq s^L \text{~and~} t^U \leqq s^U.$
	\item $T <_{LU} S \Longleftrightarrow T \leqq_{LU} S \text{~and~} T \neq S.$
\end{enumerate}
Equivalently, $T <_{LU} S$ if and only if
\begin{equation}\label{equation1,p1}
	\begin{array}{ccc}
	\left \{ \begin{array}{c}
			t^L < s^L \\ t^U \leqq s^U
		\end{array} \right.
		&
		\text{or~} \left \{ \begin{array}{c}
			t^L \leqq s^L \\ t^U < s^U
		\end{array} \right.
		&
		\text{or ~} \left \{ \begin{array}{c}
			t^L < s^L \\ t^U < s^U.
		\end{array} \right.
	\end{array}
\end{equation}

\noindent Ishibuchi and Tanaka \cite{ishibuchi} gave another ordering relation between the intervals $T$ and $S$ with lower center and lower half width (i.e., the less uncertainty), which is preferred for minimization problem.

\noindent Let $T=[t^L,t^U]$ $\in \mathbb{I}$, the center of $T$ is $t^C = \frac{1}{2}(t^L + t^U)$ and the half-width of $T$ is $t^W = \frac{1}{2}(t^U-t^L).$ Let $T = ~<t^C, t^W>$ and $S =~ <s^C, s^W>.$ then
\begin{enumerate}[label=(\roman*)]
	\item $T \leqq_{CW} S \Longleftrightarrow t^C \leqq s^C \text{~and~} t^W \leqq s^W.$
	\item $T <_{CW} S \Longleftrightarrow T\leqq_{CW} S \text{~and~} T \neq S.$
\end{enumerate}

\noindent We now consider the following multi-objective IVOP problem:
$$~\text{(MIVOP1)} ~~~~~~~~~ \text{Minimize}~~~ ~~~~~~ \phi(p) = (\phi_1(p), \phi_2(p),.....,\phi_l(p)),~~~~~~~~~~~~~~~~$$
$$~~~~~~~~~~ \text{subject to}, ~~~~~~ p \in F\subseteq E(\subseteq H),~~~~~~~~~~~~~~~~~~$$
where $\phi_s(p) : E \rightarrow \mathbb{I}$, $s \in \{1,2,...,l\},$ are interval-valued functions defined on a subset $E$ of Hadamard manifold $H$ and $F$ is a feasible set, assumed to be geodesic convex subset of $E$.

We say that $T = (T_1, T_2,...,T_r)$ is an interval-valued n-tuple if each component $T_k =[t_k^L, t_k^U]$ is a closed interval for $k \in \{1,2,3,...,r\}.$ Let $T = (T_1, T_2,..., T_r)$ and $S = (S_1, S_2,...,S_r)$ be two interval-valued n-tuples. We write $T\leqq_{LU} S$ if and only if $T_k \leqq_{LU} S_k$ for each $k \in \{1,2,3,...,r\}.$ and $T <_{LU} S$ if and only if $T_k \leqq_{LU} S_k$ for each $k \in \{1,2,3,...,r\}$ and $T_h <_{LU} S_h$ for at least one index $h$. Let $\bar{p}$ be a feasible solution of problem (MIVOP1), then $\phi(\bar{p})$ is an inter-valued n-tuple.

\begin{definition}\label{definition2.5,p1}
	\rm \cite{Wu1} If $\bar{p}$ is a feasible solution of problem (MIVOP1), then
	\begin{enumerate}[label=(\roman*)]
		\item $\bar{p}$ is a {\it type-I Pareto optimal solution (POS)} of problem (MIVOP1) if $\phi(\hat{p}) <_{LU} \phi(\bar{p})$ holds for no such $\hat{p} \in F$.
		\item $\bar{p}$ is a {\it strongly type-I POS} of problem (MIVOP1) if $\phi(\hat{p}) \leqq_{LU} \phi(\bar{p})$ holds for no such $\hat{p} \in F$.
		\item $\bar{p}$ is a {\it weakly type-I POS} of problem (MIVOP1) if $\phi_s(\hat{p}) <_{LU} \phi_s(\bar{p})$, $s \in \{1,2,...,l\},$ holds for no such $\hat{p} \in F$.
	\end{enumerate}
\end{definition}

\begin{definition}\label{definition2.6,p1}
	\rm \cite{Wu1} If $\bar{p}$ is a feasible solution of problem (MIVOP1), then
	\begin{enumerate}[label=(\roman*)]
		\item $\bar{p}$ is a {\it type-II POS} of problem (MIVOP1) if $\phi(\hat{p}) <_{CW} \phi(\bar{p})$ holds for no such $\hat{p} \in F$.
		\item $\bar{p}$ is a {\it strongly type-II POS} of problem (MIVOP1) if $\phi(\hat{p}) \leqq_{CW} \phi(\bar{p})$ holds for no such $\hat{p} \in F$.
		\item $\bar{p}$ is a {\it weakly type-II POS} of problem (MIVOP1) if $\phi_s(\hat{p}) <_{CW} \phi_s(\bar{p})$, $s \in \{1,2,...,l\}$, holds for no such $\hat{p} \in F$.
	\end{enumerate}
\end{definition}
$\vspace{0.01cm}$

{\noindent \bf Differentiation of Interval-valued Functions:}
\begin{definition} \rm \cite{chen} Let $H$ be a Hadamard manifold and $\phi : H \rightarrow \mathbb{I}$, let $T=[t^L,t^U] \in \mathbb{I}$ and $q \in H$. Then $\lim\limits_{p \rightarrow q} \phi(p) =T$, if, for every $\epsilon > 0$, there exists $\delta >0$, such that, for $0< d(p,q) < \delta$, we have $d_H(\phi(p), T) < \epsilon$.
\end{definition}

\begin{lemma}\label{lemma2.1,p1} {\rm \cite{chen}} Let $\phi:H \rightarrow \mathbb{I}$ with $\phi(p) = [\phi^L(p), ~\phi^U(p)]$ and $T = [t^L,t^U] \in \mathbb{I}$. Then, $\lim\limits_{p\rightarrow q} \phi(p) = T$ $\Leftrightarrow$ ~$\lim\limits_{p \rightarrow q} \phi^L(p) = t^L$ and $\lim\limits_{p \rightarrow q} \phi^U(p) = t^U.$
\end{lemma}

\begin{definition}\label{definition2.8,p1}
	\rm Let $H$ be a Hadamard manifold and $E \subseteq H$ be a nonempty open subset of $H$. An interval-valued function $\phi: E \rightarrow \mathbb{I}$ with $\phi(p) = [\phi^L(p), \phi^U(p)]$ is said to be weakly directionally differentiable at $p \in E$ in the direction $w \in T_p(H)$ if the real valued functions $\phi^L(p)$ and $\phi^U(p)$ are directionally differentiable at $p \in E$ in the direction of $w \in T_p(H)$. If $\phi$ is weakly directionally differentiable at $p$ for all directions $w \in T_p(H)$, we say that $\phi$ is weakly directionally differentiable at $p$. 
\end{definition}

In \cite{stefanini}, Stefanini and Bede introduced the generalized Hukuhara difference (gH-difference) of two intervals $T$ and $S$, which is expressed as
$$T \ominus_g S = U ~ \Leftrightarrow ~ \left \{ \begin{array}{c}
	\text{(i) T = S + U, or} \\ \text{(ii) S = T +(-1)U.}  
\end{array} \right.$$
In case (i), the gH-difference coincides with the H-difference \cite{Wu}. For any two intervals $T=[t^L,t^U], ~ S = [s^L, s^U], ~ T \ominus_g S $ always exists and is unique. Also, we have 
$$T\ominus_g T = [0,0] ~~~ \text{and} ~~~ T \ominus_g S = [\text{min}\{t^L-s^L, t^U-s^U\}, \text{max}\{t^L-s^L, t^U-s^U\}].$$

\begin{definition}
	\rm \cite{chen}  Let $H$ be a Hadamard manifold and $E \subseteq H$ be a nonempty open set. Then, $\phi: E \rightarrow \mathbb{I}$ with $\phi(p) = [\phi^L(p), \phi^U(p)]$ is said to be gH-directionally differentiable at $p \in E$ in the direction $w \in T_p(H)$, if, there exists $\phi'(p;w) \in \mathbb{I}$, such that 
	$$ \phi'(p;w) = \lim\limits_{\alpha \rightarrow 0^+} \frac{\phi(exp_p~\alpha w) \ominus_g \phi(p)}{\alpha}$$
	exists, where $\phi'(p;w)$ is called the gH-directional derivative of $\phi$ at $p$ in the direction of $w$. If $\phi$ is gH-directionally differentiable at $p$, for all directions $w \in T_p(H)$, we say that $\phi$ is gH-directionally differentiable at $p$.
\end{definition}

{\noindent \bf Convexity and pseudo-convexity of interval-valued functions:}

\noindent Next, we recall the notions of convexity and pseudo-convexity presented in \cite{chen}.
\begin{definition}\label{definition2.10,p1}\rm \cite{chen} Let $E \subseteq H$ be a nonempty open geodesic convex set, let $\gamma : [0,1] \rightarrow H$ be a geodesic with $\gamma (0) = p$ and $\gamma(1) = q$. Then a function $\phi : E \rightarrow \mathbb{I}$ is said to be
	\begin{enumerate}[label=(\roman*)]
		\item LU-convex at $p \in E$, if, for all $q \in E$ and $\alpha \in [0,1]$,
		\begin{equation}
			\phi(\gamma(\alpha)) \leqq_{LU}(1-\alpha)\phi(p) + \alpha \phi(q). \hspace*{3cm}
		\end{equation}
		\item CW-convex at $p \in E$, if, for all $q \in E$ and $\alpha \in [0,1]$,
		$$	\phi(\gamma(\alpha)) \leqq_{CW}(1-\alpha)\phi(p) + \alpha \phi(q).\hspace*{3cm}$$
	\end{enumerate}
\end{definition}

\begin{remark}
	\rm 	We say, $\phi$ is LU-convex (CW-convex) on $E$ if $\phi$ is LU-convex (CW-convex) at each $p \in E$.
\end{remark}

\begin{lemma}\label{lemma2.2,p1}{\rm \cite{chen}} For a function $\phi : E \rightarrow \mathbb{I}$ with $\phi(p) = [\phi^L(p), \phi^U(p)]$, $p \in E$, defined on a nonempty open geodesic convex set $E$, we have, $\phi$ is LU-convex at $p$ $\Leftrightarrow$ $\phi^L$ and $\phi^U$ are convex at $p$.
\end{lemma}

\begin{lemma}\label{lemma2.3,p1}{\rm \cite{chen}} Let $\phi_s : E \rightarrow \mathbb{I}$, $s \in \{1,2,...,l\}$, be LU-convex functions. Then, $\displaystyle \sum_{s=1}^{l} k_s\phi_s(p)$ is LU-convex, where $k_s > 0$, $s \in \{1,2,...,l\}$.
\end{lemma}	

The next lemma shows that weakly directional differentiability implies gH-directional differentiability.

\begin{lemma}\label{lemma2.5,p1} Let~ $\phi:E \rightarrow \mathbb{I}$ with $\phi(p) = [\phi^L(p), \phi^U(p)]$, defined on an open set $E$, be weakly directionally differentiable at $p\in E$ in the direction $w\in T_p(H)$, then the function $\phi(x)$ is gH-directionally differentiable at $p \in E$ in the direction $w \in T_p(H)$. Furthermore, we have
	\begin{equation} \label{eq*,p1}
		\phi'(p;w) = \big[\text{min} \{(\phi^L)'(p;w),~ (\phi^U)'(p;w)\}, ~ \text{max} \{(\phi^L)'(p;w), ~ (\phi^U)'(p;w)\}\big],
	\end{equation}
	where $(\phi^L)'(p;w)$ and $(\phi^U)'(p;w)$ denote directional derivatives of $\phi^L(p)$ and $\phi^U(p)$ at $p$ in the direction $w$, respectively.
\end{lemma}
\begin{proof}
	\begin{align*}
		\phi'(p;w) &= \lim\limits_{\alpha \rightarrow 0^+} \frac{\phi(exp_p~\alpha w)\ominus_g\phi(p)}{\alpha}\\
		&= \lim\limits_{\alpha \rightarrow 0^+}\frac{1}{\alpha}\big[\text{min}\big\{\phi^L(exp_p~\alpha w) - \phi^L(p),~ \phi^U(exp_p~\alpha w) - \phi^U(p),\\
		&~~~~~~~~~~~~~~~~~~\text{max}\big\{\phi^L(exp_p~\alpha w) - \phi^L(p),~ \phi^U(exp_p~\alpha w) - \phi^U(p)\big\}\big]\\
		&=\lim\limits_{\alpha \rightarrow 0^+}\bigg[\text{min}\bigg\{\frac{\phi^L(exp_p~\alpha w) - \phi^L(p)}{\alpha},~\frac{\phi^U(exp_p~\alpha w) - \phi^U(p)}{\alpha} \bigg\},\\
		&~~~~~~~~~~~~~~~~~\text{max}\bigg\{\frac{\phi^L(exp_p~\alpha w) - \phi^L(p)}{\alpha},~\frac{\phi^U(exp_p~\alpha w) - \phi^U(p)}{\alpha} \bigg\}\bigg],
	\end{align*}
	which, by Lemma \ref{lemma2.1,p1} and Definition \ref{definition2.8,p1}, yields
	$$\phi'(p;w) = [\text{min} \{(\phi^L)'(p;w),~ (\phi^U)'(p;w)\}, ~ \text{max} \{(\phi^L)'(p;w), ~ (\phi^U)'(p;w)\}]$$
\end{proof}

\begin{remark}
	\rm In general, the converse of Lemma \ref{lemma2.5,p1} is not true, i.e., if an IVF is gH-directionally differentiable, then the equation (\ref{eq*,p1}) may remain invalid. This is in contrast to the assertion made by chen (\cite{chen}, Lemma 3.2), which claims that gH-directionally differentiability of an IVF defined on a Hadamard manifold is equivalent to the weak directional differentiability of that function. For counter example, one can refer to \cite{hilal}.
\end{remark}

\begin{lemma}\label{lemma2.6,p1} For a function~ $\phi : E \rightarrow \mathbb{I}$, with~ $\phi(p)~= ~[\phi^L(p),~ \phi^U(p)]\\ = <\phi^C(p),~ \phi^W(p)>$, $p \in E$, defined on a nonempty open geodesic convex set $E$, we have, $\phi$ is CW-convex at $p$ $\Leftrightarrow$ $\phi^C$ and $\phi^W$ are convex at $p$.
\end{lemma}
\begin{proof}
	The result follows directly from Lemma \ref{lemma2.1,p1} and Definition \ref{definition2.10,p1}(ii).
\end{proof}
\begin{definition}\label{definition2.11,p1}{\rm \cite{chen}} \rm Let $\phi: E \rightarrow \mathbb{R}$ be a directionally differentiable function defined on a nonempty open geodesic convex subset $E \subseteq H$. Then $\phi$ is said to be
	\begin{enumerate}[label=(\roman*)]
		\item pseudo-convex at $\bar{p} \in E$ if, for any $p \in E$,
		$$\phi(p) < \phi(\bar{p}) \Rightarrow \phi'(\bar{p}; exp^{-1}_{\bar{p}} p) < 0.\hspace*{5cm}$$
		\item strictly pseudo-convex at $\bar{p} \in E$ if, for any $p \in E$,
		$$\phi(p) \leqq \phi(\bar{p}) \Rightarrow \phi'(\bar{p}; exp^{-1}_{\bar{p}} p) < 0.\hspace*{5cm}$$	
	\end{enumerate}
\end{definition}

\begin{definition}\label{definition2.12,p1}\rm Let $E \subseteq H$ be a nonempty open geodesic convex set and the interval-valued function $\phi : E \rightarrow \mathbb{I}$ with $\phi(p)=[\phi^L(p), \phi^U(p)]$ be weakly directionally differentiable at $p \in E$, then\\ 
	(i) $\phi$ is LU-pseudo-convex (LU-strictly pseudo-convex) at $\bar{p}$ $\Longleftrightarrow$ $\phi^L$ and $\phi^U$ are pseudo-convex (strictly pseudo-convex) at $\bar{p}.$\\
	(ii) $\phi$ is CW-pseudo-convex (CW-strictly pseudo-convex) at $\bar{p}$ $\Longleftrightarrow$ $\phi^C$ and $\phi^W$ are pseudo-convex (strictly pseudo-convex) at $\bar{p}.$\\
\end{definition}

\section{The KKT Optimality Conditions}
Consider $\phi : E \rightarrow \mathbb{R}$ and $\psi_u:E \rightarrow\mathbb{R}$, u=1,2,...,t, $E \subseteq H$  be a nonempty open geodesic convex set. We consider the following real valued optimization problem (RVOP) 
$$ ~~~~~~~\text{(RVOP)} ~~~~~~ \text{Minimize} ~~~~~\phi(p),~~~~~~~~~~~~~~~~~~~~~$$
$$~~~~~~~~~~~~~~~~~~~~~~~~~~~~~~~~~~\text{subject to},~~ \psi_u(p) \leqq 0,~u \in \{1,2,...,t\}.~~~~~~$$
Let $F=\{p \in E : \psi_u(p) \leqq 0, u = 1,2,...,t\}$ be the feasible region of (RVOP).

\begin{theorem}\label{theorem3.1,p1}{\rm \cite{chen}} Let $\bar{p} \in F$ and assume that the functions $\phi : E \rightarrow \mathbb{R}$ and $\psi_u : E \rightarrow \mathbb{R}$, $u \in \{1,2,...,t\}$, are convex and weakly directional differentiable at $\bar{p}$. Suppose that there exist scalars $0 \leqq \mu_u \in \mathbb{R}$, $u \in \{1,2,...,t\}$, such that,
	$$\phi'(\bar{p}; exp^{-1}_{\bar{p}}~p) + \sum_{u=1}^{t} \mu_u\psi_u'(\bar{p}; exp^{-1}_{\bar{p}}~p) \geqq 0;$$
	$$\mu_u\psi_u(\bar{p}) =0, ~~ u \in \{1,2,...,t\},$$
	then, $\bar{p}$ is an optimal solution of problem (RVOP).
\end{theorem}

\noindent Next, we consider the following multi-objective IVOP problem:
$$~~~~~ \text{(MIVOP2)} ~~~~~~ \text{Minimize} ~~~~~\phi(p) = (\phi_1(p), \phi_2(p),...,\phi_l(p)),~~~~~~~~~~~~~~~~~~~~~$$
$$~~~~\text{subject to},~~ \psi_u(p) \leqq 0,~u \in \{1,2,...,t\}~~~~~~~$$
where $\phi_s : E \rightarrow \mathbb{I}$ with $\phi_s(p) = [\phi_s^L(p), \phi_s^U(p)]$, $s \in \{1,2,...,l\}$, are interval-valued functions defined on a geodesic convex set $E$ and $\psi_u : E \rightarrow \mathbb{R}$, u=1,2,...,t, are assumed to be real-valued convex functions on $E$. It follows that the problem (MIVOP1) and (MIVOP2) become same by considering the geodesic convex set $F$ as $F =\{p : \psi_u(p) \leqq 0, ~~ u = 1,2,...,t\}.$

\subsection{KKT Conditions for Type-I and Type-II POS}
In this section, we discuss some KKT conditions for type-I POS by using the concepts of LU-convexity and LU-pseudo-convexity as well as some conditions for type-II POS by using CW-convexity and  CW-pseudo-convexity for (MIVOP2).
\begin{theorem}\label{theorem3.2,p1}
	Let $\psi_u: E \rightarrow \mathbb{R}$, $u \in \{1,2,...,t\}$, be convex at $\bar{p} \in F$.
	\begin{enumerate}[label=\bf (\alph*)]
		\item Suppose that the interval-valued functions $\phi_s: E \rightarrow \mathbb{I}$ with $\phi_s(p) = [\phi_s^L(p), \phi_s^U(p)]$, $s \in \{1,2,...,l\}$, are LU-convex and weakly directionally differentiable at $\bar{p}$ and that, there exist, $0 < \lambda_s^L,~ \lambda_s^U \in \mathbb{R}$, $s \in \{1,2,...,l\}$, and $0 \leqq \mu_u \in \mathbb{R}$, $u \in \{1,2,...,t\}$, such that:
		\begin{enumerate}[label=(\roman*)]
			\item $\displaystyle \sum_{s=1}^{l} \lambda_s^L (\phi_s^L)'(\bar{p};~ exp^{-1}_{\bar{p}} ~ p) + \displaystyle \sum_{s=1}^{l} \lambda_s^U (\phi_s^U)'(\bar{p};~ exp^{-1}_{\bar{p}} ~ p) + \displaystyle\sum_{u=1}^{t} \mu_u \psi_u'(\bar{p};~ exp^{-1}_{\bar{p}} ~ p) \geqq 0;$
			\item $\mu_u\psi_u(\bar{p}) = 0, ~ u \in \{1,2,...,t\},$
		\end{enumerate}
	\end{enumerate}
	then, $\bar{p}$ is a type-I POS of (MIVOP2).
	\begin{enumerate}[resume*]
		\item  Suppose that the interval-valued functions $\phi_s: E \rightarrow \mathbb{I}$ with $\phi_s(p) = [\phi_s^L(p), \phi_s^U(p)]$, $s \in \{1,2,...,l\}$, are CW-convex and weakly directionally differentiable at $\bar{p}$ and that, there exist, $0 < \lambda_s^C,~ \lambda_s^W \in \mathbb{R}$, $s \in \{1,2,...,l\}$, and $0 \leqq \mu_u \in \mathbb{R}$, $u \in \{1,2,...,t\}$, such that:
		\begin{enumerate}
			\item[(iii)] $\displaystyle \sum_{s=1}^{l} \lambda_s^C (\phi_s^C)'(\bar{p};~exp^{-1}_{\bar{p}} p) + \displaystyle \sum_{s=1}^{l} \lambda_s^W (\phi_s^W)'(\bar{p};~ exp^{-1}_{\bar{p}} p)+ \displaystyle\sum_{u=1}^{t} \mu_u \psi_u'(\bar{p};~ exp^{-1}_{\bar{p}} p) \geqq 0;$
			\item[(iv)] $\mu_u\psi_u(\bar{p}) = 0$, ~ $u \in \{1,2,...,t\}$,
		\end{enumerate} 
	\end{enumerate}
	then, $\bar{p}$ is a type-II POS of (MIVOP2).
	\begin{enumerate}[resume*]
		\item Suppose that the interval-valued functions $\phi_s: E \rightarrow \mathbb{I}$ with $\phi_s(p) = [\phi_s^L(p), \phi_s^U(p)]$, $s \in \{1,2,...,l\}$, are CW-convex and weakly directionally differentiable at $\bar{p}$ and that, there exist, $0 < \lambda_s^L,~ \lambda_s^U \in \mathbb{R}$, $s \in \{1,2,...,l\}$, and $0 \leqq \mu_u \in \mathbb{R}$, $u \in \{1,2,...,t\}$, such that:
		\begin{enumerate}
			\item[(v)] $\lambda_s^L < \lambda_s^U$, ~ $s \in \{1,2,...,l\}$;
			\item[(vi)] $\displaystyle \sum_{s=1}^{l} \lambda_s^L (\phi_s^L)'(\bar{p};~ exp^{-1}_{\bar{p}} ~ p) + \displaystyle \sum_{s=1}^{l} \lambda_s^U (\phi_s^U)'(\bar{p};~ exp^{-1}_{\bar{p}} ~ p) + \displaystyle\sum_{u=1}^{t} \mu_u \psi_u'(\bar{p};~ exp^{-1}_{\bar{p}} ~ p) \geqq 0;$
			\item[(vii)] $\mu_u\psi_u(\bar{p}) = 0$, ~ $u \in \{1,2,...,t\}$,
		\end{enumerate}
	\end{enumerate}
	then, $\bar{p}$ is a type-II POS of (MIVOP2).
\end{theorem}
\begin{proof}
	Define, $\bar{\phi} : E \rightarrow \mathbb{R}$ as follows
	\begin{equation}
		\bar{\phi}(p) = \sum_{s=1}^{l} \lambda_s^L \phi^L_s(p) + \sum_{s=1}^{l} \lambda_s^U \phi^U_s(p).
	\end{equation}
	Since, each $\phi_s(p)$, $s=1,2,...l$, is LU-convex function. It follows from Lemma \ref{lemma2.2,p1} and Definition \ref{definition2.8,p1} that $\phi_s^L(p)$ and $\phi_s^U(p)$ are convex functions and directionally differentiable at $\bar{p}$. From Lemma \ref{lemma2.3,p1}, $\bar{\phi}(p)$ is a convex function and directionally differentiable at $\bar{p}$. Since,
	$$\bar{\phi}'(\bar{p}; exp^{-1}_{\bar{p}} ~p) = \sum_{s=1}^{l}\lambda_s^L (\phi^L_s)'(\bar{p}; exp^{-1}_{\bar{p}} ~p) + \sum_{s=1}^{l} \lambda_s^U (\phi^U_s)'(\bar{p}; exp^{-1}_{\bar{p}} ~p).$$
	From (i), we have \\
	$$(\bar{\phi})'(\bar{p}; exp^{-1}_{\bar{p}} ~p) + \sum_{u=1}^{t} \mu_u \psi_u'(\bar{p}; exp^{-1}_{\bar{p}} ~p) \geqq 0, ~ \text{for all}~ p \in E;$$
	$$\mu_u\psi_u(\bar{p}) = 0, ~~~  u \in \{1,2,...,t\}.$$ 
	This together with Theorem \ref{theorem3.1,p1}, imply that $\bar{p}$ is an optimal solution of $\bar{\phi}(p)$ subjected to constraints of (MIVOP2); so, 
	\begin{equation}\label{equation4,p1}
		\bar{\phi}(\bar{p}) \leqq \bar{\phi}(p), ~ \text{for all} ~ p \in E.
	\end{equation}
	Now, let on contrary that $\bar{p}$ is not a type-I POS of (MIVOP2), then from Definition \ref{definition2.5,p1}(i), there exists $\hat{p} \in F$ and $c \in \{1,2,...,l\}$ such that
	$$\phi_c(\hat{p}) <_{LU} \phi_c (\bar{p}),$$
	i.e., from (\ref{equation1,p1}), we get
	\begin{equation}\label{equation5,p1}
		\begin{array}{ccc}
			\begin{array}{c}
				\phi_c^L(\hat{p}) < \phi_c^L(\bar{p})\\ \phi_c^U(\hat{p}) \leqq \phi_c^U(\bar{p})
			\end{array}
			&
			\text{or~} \left \{ \begin{array}{c}
				\phi_c^L(\hat{p}) \leqq \phi_c^L(\bar{p})\\ \phi_c^U(\hat{p}) < \phi_c^U(\bar{p})
			\end{array} \right.
			&
			\text{or ~} \left \{ \begin{array}{c}
				\phi_c^L(\hat{p}) < \phi_c^L(\bar{p})\\ \phi_c^U(\hat{p}) < \phi_c^U(\bar{p}).
			\end{array} \right.
		\end{array}
	\end{equation}
	For $s \neq c$, $\phi_s(\hat{p}) \leqq_{LU} \phi_s(\bar{p})$, i.e.,
	\begin{equation}\label{equation6,p1}
		\phi_s^L(\hat{p}) \leqq \phi_s^L(\bar{p}) \text{~~and~~} \phi_s^U(\hat{p}) \leqq \phi_s^U(\bar{p}).
	\end{equation}
	Since $\lambda_s^L, ~ \lambda_s^U > 0$, $s \in \{1,2,...,l\}$, and using (\ref{equation5,p1}) and (\ref{equation6,p1}), we have 
	$$\sum_{s=1}^{l} \lambda_s^L \phi^L_s(\hat{p}) + \sum_{s=1}^{l} \lambda_s^U \phi^U_s(\hat{p}) < \sum_{s=1}^{l} \lambda_s^L \phi^L_s(\bar{p}) + \sum_{s=1}^{l} \lambda_s^U \phi^U_s(\bar{p}).$$
	$$\text{i.e.,} ~~~~~ \bar{\phi}(\hat{p}) < \bar{\phi}(\bar{p}),$$
	which contradicts with (\ref{equation4,p1}). Hence, we conclude that $\bar{p}$ is a type-I POS of (MIVOP2).\\\\
	For part {\bf (b)}, we define
	$$ \bar{\phi}(p) = \sum_{s=1}^{l} \lambda_s^C \phi^C_s(p) + \sum_{s=1}^{l} \lambda_s^W \phi^W_s(p).$$
	The rest follows same as that of result {\bfseries (a)} and using lemma \ref{lemma2.6,p1}.\\\\
	For part {\bfseries (c)}, we assume\\
	$$\lambda_s^L = \frac{\lambda_s^C - \lambda_s^W}{2} ~~ \text{and}~~ \lambda_s^U = \frac{\lambda_s^C + \lambda_s^W}{2}.$$
	After simple calculations and using condition $(v)$, we obtain 
	$$\lambda_s^C = \lambda_s^L + \lambda_s^U >0 ~~ \text{and} ~~ \lambda_s^W = \lambda_s^U - \lambda_s^L >0.$$
	The result thus follows immediately from {\bf (b)}.
\end{proof}

We now present an example of an optimization problem with multi-objective interval-valued functions in which some of the objective functions and constraints fail to be LU-convex in linear space and hence the results in \cite{ycc}, \cite{eh},\cite{Wu}, \cite{Wu1}, \cite{jz} cannot be applied to such problem. However, the same functions happen to be LU-convex in a Riemannian manifold under a suitable Riemannian metric.

\begin{example}\label{example3.1,p1}
	\rm	Let $H = \mathbb{R}^2_{++} :=\{(p_1,~p_2) \in \mathbb{R}^2 : p_1, ~p_2 >0\}$ be a Riemannian manifold with Riemannian metric $<,> ~=~ <Q(p)w,~z>$ for any pair $(z,w) \in T_p(H) \times T_p(H)$, where $Q(p) = (g_{ij}(p))$ defines a $2 \times 2$ matrix $Q$, given by $g_{ij}(p) = (\delta_{ij}/p_ip_j)$.\\
	The Riemannian distance $d: H \times H \rightarrow \mathbb{R}^+$ between any two points $p = (p_1, p_2)$ and $q = (q_1,q_2)$ is given by\\
	$$d(p,q) = \Vert \big (ln \frac{p_1}{q_1}, ln \frac{p_2}{q_2}\big ) \Vert.$$
	We know that the sectional curvature of H is identically 0 and H is a Hadamard manifold. For more details see \cite{bento}.
	The geodesic curve $\gamma : \mathbb{R} \rightarrow H$ satisfying $\gamma (0) = p \in M$ and $\gamma'(0) = w \in T_p(H) = \mathbb{R}^2$ is given by\\
	$$\gamma (\alpha) = \big (p_1e^{\frac{w_1}{p_1}\alpha},~p_2e^{\frac{w_2}{p_2}\alpha} \big).$$
	For any $p=(p_1,p_2) \in H$ and any $w=(w_1,w_2) \in T_p(H) = \mathbb{R}^2$, the exponential map $exp_p : T_p(H) \rightarrow H$ is given by\\
	$$exp_p(w) = \gamma (1) = \big (p_1e^{\frac{w_1}{p_1}},~p_2e^{\frac{w_2}{p_2}} \big). $$
	For any $p= (p_1,p_2)$, $q = (q_1,q_2)$ $\in H$ and $w = (w_1,w_2) \in T_p(H)$, the inverse exponential map is
	$$exp_{(p_1,p_2)}^{-1} (q_1,q_2) = (w_1,w_2) = \big (p_1 ln \frac{q_1}{p_1}, ~ p_2ln \frac{q_2}{p_2} \big ).$$
	Let $\phi_s: H \rightarrow \mathbb{I}$, $s \in \{1,2\}$, be interval-valued functions and $\psi_u: H \rightarrow \mathbb{R}$, $u \in \{1,2,3,4,5\}$, be the real valued constraint functions for the following.
	\begin{align*}
		\text{(MIVOP3)}~~\text{Minimize~~~~~} \phi(p) &= (\phi_1(p), \phi_2(p)),\\ 
		\intertext{\hspace{4.3cm}where ~~~$\phi_1(p)= \big ( [ln(p_1) + 3, ~ln(p_1) + 5];$} 
		\intertext{\hspace{4.5cm}and ~~~~$\phi_2(p)=[p_1^2 + p_2^2 + 2,~ p_1^2 + p_2^2 + 7] \big ),$}
		\text{subject to,~~~} \psi_1(p) &= ln(p_1) + \sqrt{p_2} -1 \leqq 0,\\
		\psi_2(p) &= -ln(p_1)\leqq 0,\\
		\psi_3(p)&=p_1+p_2-2 \leqq 0,\\
		\psi_4(p)&=p_1^2p_2 -7 \leqq 0,\\
		\psi_5(p)&=-ln(p_2) \leqq 0.
	\end{align*}
	It is easy to see that $\phi_1(p)$ and $\phi_2(p)$ are LU-convex functions and gH-directionally differentiable on $H$ but $\phi_1(p)$ fails to be an LU-convex function in the usual sense. The functions $\psi_u(x)$, $u \in \{1,2,3,4,5\}$, are directionally differentiable convex functions on M but $\psi_4(p)$ is not a convex functions in the usual sense.\\
	Now, for $\bar{p} = (1,1)$ and for any $p=(p_1,p_2) \in H$, we have the following
	\begin{align*}
		(\phi_1^L)'(\bar{p}; ~ exp_{\bar{p}}^{-1} ~p) &= ln(p_1),\\
		(\phi_1^U)'(\bar{p}; ~ exp_{\bar{p}}^{-1} ~p) &= ln(p_1),\\
		(\phi_2^L)'(\bar{p}; ~ exp_{\bar{p}}^{-1} ~p) &= 2ln(p_1) +2ln(p_2),\\
		(\phi_2^U)'(\bar{p}; ~ exp_{\bar{p}}^{-1} ~p) &= 2ln(p_1) +2ln(p_2),\\
		\psi_1'(\bar{p}; ~ exp_{\bar{p}}^{-1} ~p) &= ln(p_1) + \frac{1}{2}ln(p_2),\\
		\psi_2'(\bar{p}; ~ exp_{\bar{p}}^{-1} ~p) &= -ln(p_1),\\
		\psi_3'(\bar{p}; ~ exp_{\bar{p}}^{-1} ~p) &= ln(p_1) + ln(p_2),\\
		\psi_4'(\bar{p}; ~ exp_{\bar{p}}^{-1} ~p) &= 2ln(p_1) + ln(p_2),\\
		\psi_5'(\bar{p}; ~ exp_{\bar{p}}^{-1} ~p) &= -ln(p_2).
	\end{align*}
	It is easy to check that conditions (i) and (ii) in Theorem \ref{theorem3.2,p1}(a) hold at $\bar{p} = (1,1)$ with the Lagrangian multipliers as follows:
	 $$(\lambda_1^L, \lambda_1^U, \lambda_2^L, \lambda_2^U, \mu_1, \mu_2, \mu_3, \mu_4, \mu_5)=(1,1,1,1,1,7,0,0, \frac{9}{2}).$$ 
	 So, by Theorem \ref{theorem3.2,p1}(a), we conclude that $\bar{p} = (1,1)$ is a type-I POS of (MIVOP3).
\end{example}

The number of multipliers involved in Theorem \ref{theorem3.2,p1} can be reduced very easily, for this let's recall the decomposition introduced by Wu \cite{Wu1}. Assume that $t \geqq 2l$, $t$ and $l$ are positive integers. Then the set $\{1,2,...,t\}$ can be decomposed into $2l$ nonempty subsets $Q_i$ for $i \in \{1,2,...,2l\}$ such that $Q_i \cap Q_j = \emptyset$ for $i \neq j$ and $\cup_{i=1}^{2l} Q_i = \{1,2,...,t\}$. This decomposition is denoted by $Q_{t,2l} = \{Q_1, Q_2,...,Q_{2l}\}$.

\begin{theorem}\label{theorem3.3,p1}
	Let $\psi_u:E \rightarrow \mathbb{R}$ be convex on $E \subseteq H$, $u \in \{1,2,...,t\}$, and let $\bar{p} \in F = \{p \in E: \psi_u(p)\leqq 0\}$. Suppose, $\phi_s : E \rightarrow \mathbb{I}$ with $\phi_s(p) = [\phi_s^L(p), \phi_s^U(p)]$ are LU-convex and weakly directionally differentiable at $\bar{p}$ for $s \in \{1,2,...,l\}$. Further, assume $t \geqq 2l$ with a decomposition $Q_{t,2l}$. If, there exist, $0 \leqq \mu_u \in \mathbb{R}$, $u \in \{1,2,...,t\}$, such that:
	\begin{enumerate}[label=(\roman*)]
		\item $(\phi_s^L)'(\bar{p}; exp^{-1}_{\bar{p}} ~ p)$ +  $\displaystyle\sum_{u \in Q_s} \mu_u \psi_u'(\bar{p}; exp^{-1}_{\bar{p}} ~ p) \geqq 0;$
		\item $(\phi_s^U)'(\bar{p}; exp^{-1}_{\bar{p}} ~ p)$ +  $\displaystyle\sum_{u \in Q_{s+l}} \mu_u \psi_u'(\bar{p}; exp^{-1}_{\bar{p}} ~ p) \geqq 0;$
		\item $\mu_u\psi_u(\bar{p}) = 0$, ~ $u \in \{1,2,...,t\}$,
	\end{enumerate}
	then, $\bar{p}$ is a type-I POS of (MIVOP2).
\end{theorem}
\begin{proof}
	Let $0 < \lambda_s^L, \lambda_s^U \in \mathbb{R}$. Then, for all $p \in F$, the conditions (i) and (ii) can be rewritten as
	\begin{align*}
		\hspace{-1.5cm}(a)~ \lambda_s^L(\phi_s^L)'(\bar{p}; exp^{-1}_{\bar{p}} ~ p) + \lambda_s^U(\phi_s^U)'(\bar{p}; exp^{-1}_{\bar{p}} ~ p) &+  \displaystyle\sum_{u \in Q_s} \bar{\mu}_u^L \psi_u'(\bar{p}; exp^{-1}_{\bar{p}} ~ p)~~~~~~~~~~~~~\\  &+ \displaystyle\sum_{u \in Q_{s+l}} \bar{\mu}_u^U \psi_u'(\bar{p}; exp^{-1}_{\bar{p}} ~ p) \geqq 0,~~s \in \{1,2,...,l\}.
	\end{align*}
	(b) $\bar{\mu}_u^L\psi_i(\bar{p}) = 0, ~~u \in Q_s$ and $\bar{\mu}_u^U \psi_u(\bar{p}) = 0, ~~u \in Q_{s+l}$, ~ $s \in \{1,2,...,l\}$,\\\\
	where $\bar{\mu}_u^L = \lambda_s^L \mu_u, ~~ u \in Q_s$ and $\bar{\mu}_u^U = \lambda_s^U \mu_u, ~~ u \in Q_{s+l}$.\\\\
	Taking summation over $s$ in (a), we have
	\begin{align*}
		\displaystyle \sum_{s=1}^{l} \lambda_s^L(\phi_s^L)'&(\bar{p}; exp^{-1}_{\bar{p}} ~ p) + \displaystyle \sum_{s=1}^{l} \lambda_s^U(\phi_s^U)'(\bar{p}; exp^{-1}_{\bar{p}} ~ p)\\ 
		&+ \displaystyle \sum_{s=1}^{l} \Big [\displaystyle\sum_{u \in Q_s} \bar{\mu}_u^L \psi_u'(\bar{p}; exp^{-1}_{\bar{p}} ~ p) + \displaystyle\sum_{u \in Q_{s+l}} \bar{\mu}_u^U \psi_u'(\bar{p}; exp^{-1}_{\bar{p}} ~ p)\Big ] \geqq 0.
	\end{align*}
	The result now follows by taking into consideration the Theorem \ref{theorem3.2,p1}(a) and the decomposition.
\end{proof}

\begin{remark}
	\rm Parts {\bf(b)} and {\bf (c)} of Theorem \ref{theorem3.2,p1} can also be used in the similar way for Theorem \ref{theorem3.3,p1}.
\end{remark}

Now, we consider the weaker condition of pseudo-convexity for the objective functions involved in the optimization problem (MIVOP2).
But, first we present an example of a function defined on a Hadamard manifold which is pseudo-convex function but fails to be a convex function. 

\begin{example}\label{example3.2,p1}
	\rm 	Let $H = \mathbb{R}^2_{++} :=\{(p_1,~p_2) \in \mathbb{R}^2 : p_1, ~p_2 >0\}$ be a Riemannian manifold with Riemannian metric as defined in Example \ref{example3.1,p1}.
	
	Define a function $\phi: H \rightarrow \mathbb{R}$ as follows
	$$\phi((p_1, p_2)) = \frac{(ln(p_1))^2 + (ln(p_2))^2}{1 + (ln(p_1))^2 + (ln(p_2))^2}.$$
	The geodesic $\gamma(t),~ t \in [0, \infty]$, emanating from $\gamma(0) =\bar{p}=(1,1)$ in the direction $\gamma'(0)=w=(w_1, w_2) \in \mathbb{R}^2$ is given by
	$$\gamma(t) = (e^{w_1t}, e^{w_2t}).$$
	First, we show $\phi$ is strictly pseudo-convex at $\bar{p}$ by using the contrapositive statement of Definition \ref{definition2.11,p1}(ii). From simple calculations, we have
	$$(\phi \circ \gamma)(t) = \frac{(w_1^2 + w_2^2)t^2}{1 + (w_1^2 + w_2^2)t^2}.$$
	The directional derivative of $\phi$ at $\bar{p} = (1,1)$ in the direction $w = exp^{-1}_{\bar{p}} p \in \mathbb{R}^2$,  $p = (p_1,p_2) \in H$, is given by
	$$\phi'(\bar{p};~w) = \frac{d}{dt}(\phi \circ \gamma)(t)|_{t=0}=0.$$
	Also, we have 
	$$\phi(\bar{p}) = 0 \text{   and   } \phi(p) = \frac{(ln(p_1))^2 + (ln(p_2))^2}{1 + (ln(p_1))^2 + (ln(p_2))^2}.$$
	Since, $\phi'(\bar{p};~ exp^{-1}_{\bar{p}} p) = 0$ for every $p \in H$ and $\phi(p) > \phi(\bar{p})$ for every $p (\neq \bar{p}) \in H$, we conclude from Definition \ref{definition2.11,p1} (ii) (using the contrapositive statement of Definition) that $\phi$ is strictly pseudo-convex at $\bar{p}=(1,1)$.
	
	Next, we show $\phi$ is not a convex function at $\bar{p}$. For this, we choose $\hat{p}=(e^2,e^2)$ and $t=\frac{1}{2}$. By simple calculations, we have
	$$(\phi \circ \gamma)(\frac{1}{2}) = 0.666, ~\phi(\bar{p}) = 0~ \text{and}~ \phi(\hat{p})=0.888.$$
	It is now easy to see the following
	$$(\phi \circ \gamma)(\frac{1}{2}) > \frac{1}{2}\phi(\bar{p}) + \frac{1}{2}\phi(\hat{p}).$$ 
	Hence, $\phi$ is not a convex function at $\bar{p}=(1,1).$  \\
\end{example}

Let $F$ be the nonempty feasible set and $\bar{p} \in \bar{F}$, where $\bar{F}$ is the closure of $F$ and let $p \in F$, then the {\it set of geodesic feasible directions} of $F$ at $p$ is given by
$$A_pE = \{w \in T_p(H): \exists ~ \bar{t}>0, ~\text{such that}~ exp_p~(\alpha w) \in E, ~ \text{for all} ~\alpha \in [0, \bar{t})\}.$$

\begin{lemma}\label{lemma3.1,p1}
	{\rm \cite{chen}} Let $\bar{p} \in E$ and $\psi_u$ be directionally differentiable functions at $\bar{p}$, for all $u \in \{1,2,...,t\} = \hat{U}$. Let $J(\bar{p}) = \{u \in \hat{U}: \psi_u(\bar{p}) = 0\}$ be the index set for all binding, or active constraints of the optimization problem (MIVOP2). Let the set $G_0$ be defined as 
	$$G_0 = \{w \in T_{\bar{p}}(H) : \psi_u'(\bar{p};w) < 0, ~ \text{for all} ~ u \in J(\bar{p})\}.$$
	If $\psi_u,~ u \in J(\bar{p}),$ are strictly pseudo-convex functions at $\bar{p}$, then $A_{\bar{p}}E = G_0$.
\end{lemma}


\noindent In view of Lemma \ref{example3.1,p1}, we now prove the following result which involves the weaker condition of pseudo-convexity for the objective functions in (MIVOP2).

\begin{theorem}\label{theorem3.4,p1}
	Let $\bar{p} \in F$ and assume that the real valued constraints  $\psi_u, ~ u \in J(\bar{p}),$ of problem (MIVOP2) be strictly pseudo-convex functions at $\bar{p}$. Let $\phi_s : E \rightarrow \mathbb{I}$ with $\phi_s(p) = [\phi_s^L(p), \phi_s^U(p)]$, $s \in \{1,2,...,l\}$, defined on an open geodesic convex set $E$, be weakly directionally differentiable at $\bar{p}$.
	\begin{enumerate}[label=\bf(\alph*)]
		\item If $\phi_s$, $s \in \{1,2,...,l\}$, are LU-strictly pseudo-convex functions at $\bar{p}$ and for every feasible point $p$, there exist,~ $0 < \lambda_s^L,~ \lambda_s^U \in \mathbb{R}$, $s \in \{1,2,...,l\}$, and $0 \leqq \mu_u^L,~ \mu_u^U \in \mathbb{R}$, $u \in \{1,2,...,t\}$, such that:
		\begin{enumerate}[label=(\roman*)]
			\item $\displaystyle \sum_{s=1}^{l} \lambda_s^L (\phi_s^L)'(\bar{p}; exp^{-1}_{\bar{p}} ~ p)$ + $\displaystyle\sum_{u=1}^{t} \mu_u^L \psi_u'(\bar{p}; exp^{-1}_{\bar{p}} ~ p) \geqq 0;$
			\item $\displaystyle \sum_{s=1}^{l} \lambda_s^U (\phi_s^U)'(\bar{p}; exp^{-1}_{\bar{p}} ~ p)$ + $\displaystyle\sum_{u=1}^{t} \mu_u^U \psi_u'(\bar{p}; exp^{-1}_{\bar{p}} ~ p) \geqq 0;$
			\item $\mu_u^L\psi_u(\bar{p}) = 0 = \mu_u^U\psi_u(\bar{p})$, $u \in \{1,2,...,t\}$,
		\end{enumerate}
	\end{enumerate}
	then, $\bar{p}$ is a type-I POS of (MIVOP2).
	\begin{enumerate}[resume*]
		\item If $\phi_s$, $s \in \{1,2,...,l\}$, are CW-strictly pseudo-convex functions at $\bar{p}$ and for every feasible point $p$, there exist, $0 < \lambda_s^C,~ \lambda_s^W \in \mathbb{R}$, $s \in \{1,2,...,l\}$, and $0 \leqq \mu_u^C,~ \mu_u^W \in \mathbb{R}$, $u \in \{1,2,...,t\}$, such that:
		\begin{enumerate}[label=(\roman*)]
			\item[(iv)] $\displaystyle \sum_{s=1}^{l} \lambda_s^C (\phi_s^C)'(\bar{p}; exp^{-1}_{\bar{p}} ~ p)$ + $\displaystyle\sum_{u=1}^{t} \mu_u^C \psi_u'(\bar{p}; exp^{-1}_{\bar{p}} ~ p) \geqq 0;$
			\item[(v)] $\displaystyle \sum_{s=1}^{l} \lambda_s^W (\phi_s^W)'(\bar{p}; exp^{-1}_{\bar{p}} ~ p)$ + $\displaystyle\sum_{u=1}^{t} \mu_u^W \psi_u'(\bar{p}; exp^{-1}_{\bar{p}} ~ p) \geqq 0;$
			\item[(vi)] $\mu_u^C\psi_u(\bar{p}) = 0 = \mu_u^W\psi_u(\bar{p})$, $u \in \{1,2,...,t\}$,
		\end{enumerate}
	\end{enumerate}
	then, $\bar{p}$ is a type-II POS of (MIVOP2).
\end{theorem}

\begin{proof}
	Suppose, on contrary, $\bar{p}$ be not a type-I POS of (MIVOP2), then, there exists $\hat{p} \neq \bar{p} \in F$, such that $\phi(\hat{p}) <_{LU} \phi(\bar{p})$, i.e., there exists $c \in \{1,2,...,l\}$ such that $\phi_c(\hat{p}) <_{LU} \phi_c(\bar{p})$ $\Rightarrow$ $\phi_c^L(\hat{p}) < \phi_c^L(\bar{p})$ or $\phi_c^U(\hat{p}) < \phi_c^U(\bar{p})$.\\
	Since, $\phi_c$ is LU-strictly pseudo-convex function at $\bar{p}$, from Definition \ref{definition2.12,p1}(i), $\phi_c^L$ and $\phi_c^U$ are strictly pseudo-convex functions, hence pseudo-convex functions. From definition \ref{definition2.11,p1}(i), we have
	$$(\phi_c^L)'(\bar{p}; exp^{-1}_{\bar{p}} ~ \hat{p}) < 0 ~~~ \text{or} ~~~ (\phi_c^U)'(\bar{p}; exp^{-1}_{\bar{p}} ~ \hat{p}) < 0.$$
	Without the loss of generality, we suppose 
	\begin{equation}\label{equation7,p1}
		(\phi_c^L)'(\bar{p}; exp^{-1}_{\bar{p}} ~ \hat{p}) < 0.
	\end{equation}
	For $s \neq c$,
	$$\phi_s^L(\hat{p}) < \phi_s^L(\bar{p}) ~~ \text{or} ~~ \phi_s^L(\hat{p}) \leqq \phi_s^L(\bar{p})$$
	both of which by Definition \ref{definition2.11,p1}, yield
	\begin{equation}\label{equation8,p1}
		(\phi_s^L)'(\bar{p}; exp^{-1}_{\bar{p}} ~ \hat{p}) < 0.
	\end{equation}
	From equations (\ref{equation7,p1}) and (\ref{equation8,p1}), we have 
	\begin{equation}\label{equation9,p1}
		(\phi_s^L)'(\bar{p}; exp^{-1}_{\bar{p}} ~ \hat{p}) < 0, ~ s \in \{1,2,...,l\}.
	\end{equation}
	For, $0 < \lambda_s^L \in \mathbb{R}$, $s \in \{1,2,...,l\}$, the system (\ref{equation9,p1}) yields\\
	\begin{equation}\label{equation10,p1}
		\displaystyle \sum_{s=1}^{l} \lambda_s^L (\phi_s^L)'(\bar{p}; exp^{-1}_{\bar{p}} ~ \hat{p}) < 0.
	\end{equation}
	Let $w = exp^{-1}_{\bar{p}} ~ \hat{p}$. Since, $E$ is geodesic convex set, $p = exp_{\bar{p}}~\alpha w \in E$, for $\alpha \in (0,1)$, which shows that $w \in A_{\bar{p}} E$. From Lemma \ref{lemma3.1,p1}, we get 
	$$\psi_u'(\bar{p}; ~w) < 0, ~~~ u \in J(\bar{p}).$$
	\begin{equation}\label{equation11,p1}
		\text{i.e.,}~~~~~	\psi_u'(\bar{p}; exp^{-1}_{\bar{p}} ~ \hat{p}) < 0, ~~~ u \in J(\bar{p}).
	\end{equation}
	For,  $0 \leqq \mu_u^L \in \mathbb{R}$ and some $\mu_u^L \neq 0$, $u \in J(\bar{p})$, the system (\ref{equation11,p1}) yields
	\begin{equation}\label{equation12,p1}
		\displaystyle\sum_{u \in J(\bar{p})} \mu_u^L \psi_u'(\bar{p}; exp^{-1}_{\bar{p}} ~ \hat{p}) < 0.
	\end{equation}
	Note that,
	\begin{equation}\label{equation13,p1}
		\displaystyle\sum_{u \in J(\bar{p})} \mu_u^L \psi_u'(\bar{p}; exp^{-1}_{\bar{p}} ~ \hat{p}) = 	\displaystyle\sum_{u=1}^{t} \mu_u^L \psi_u'(\bar{p}; exp^{-1}_{\bar{p}} ~ \hat{p}), 
	\end{equation}
	with $\mu_u^L \psi_i(\bar{p}) = 0; ~$ $u \in \{1,2,...,t\}$.\\
	From inequalities (\ref{equation10,p1}) and (\ref{equation12,p1}), and equation (\ref{equation13,p1}), we have
	$$\displaystyle \sum_{s=1}^{l} \lambda_s^L (\phi_s^L)'(\bar{p}; exp^{-1}_{\bar{p}} ~ p) + \displaystyle\sum_{u=1}^{t} \mu_u^L \psi_u'(\bar{p}; exp^{-1}_{\bar{p}} ~ p) < 0,$$
	which violates condition (i).
	This proves {\bf (a)}.\\
	Part {\bf (b)} can similarly be proved as that of {\bf (a)}. 
\end{proof}

\noindent The following example is in support of Theorem \ref{theorem3.4,p1}(a) which involves the weaker condition of pseudo-convexity for the objective functions in (MIVOP2). 

\begin{example}
	\rm 	Let $H = \mathbb{R}^2_{++} :=\{(p_1,~p_2) \in \mathbb{R}^2 : p_1, ~p_2 >0\}$ be a Riemannian manifold with Riemannian metric as defined in Example \ref{example3.1,p1}.
	
	Let $\phi_s: H \rightarrow \mathbb{I}$, $s \in \{1,2\}$, be interval-valued functions and $\psi_u: H \rightarrow \mathbb{R}$, $u \in \{1,2,3,4,5\}$, be the real valued constraint functions for the following.
	\begin{align*}
		\text{Minimize~~~~~} \phi(p) &= (\phi_1(p),~ \phi_2(p)), \hspace{4.7cm}\text{(MIVOP4)}\\
		\text{where }~~~ \phi_1(p) &= [\frac{(ln(p_1))^2 + (ln(p_2))^2}{1 + (ln(p_1))^2 + (ln(p_2))^2}, \frac{(ln(p_1))^2 + (ln(p_2))^2}{1 + (ln(p_1))^2 + (ln(p_2))^2} + 1],\\
		\text{and}~~~~ \phi_2(p) &= [(ln(p_1))^2 + p_2, ~ (ln(p_1))^2 + p_2+1].\\
		\text{subject to,~~~} \psi_1(p) &= (ln(p_1))^2 + (ln(p_2))^2 -1 \leqq 0,\\
		\psi_2(p) &= p_1(ln(p_2))^2 + p_2(ln(p_1))^2 - 3\leqq 0,\\
		\psi_3(p)&= \frac{1}{p_1} + (ln(p_2))^2 -1 \leqq 0,\\
		\psi_4(p)&=p_1^2 + p_2^2 - 2 \leqq 0,\\
		\psi_5(p)&= (ln(p_1))^2 + \frac{1}{p_2} -1 \leqq 0.
	\end{align*}
	Here, both the objective functions $\phi_1(p)$ and $\phi_2(p)$ are weakly directionally differentiable and strictly LU-pseudo-convex at $\bar{p}=(1,1)$ and $\phi_1(p)$ is not LU-convex at $\bar{p}=(1,1)$. Also, the constraints $\psi_u(p)$, $u \in \{1,2,3,4,5\}$ are all strictly pseudo-convex functions at $\bar{p}=(1,1)$.
	
	The directional derivatives of the corresponding functions involved in the optimization problem at $\bar{p}=(1,1)$ in any direction $w= exp^{-1}_{\bar{p}} ~p \in \mathbb{R}^2$, $p=(p_1,p_2)\in H$, are given as follows
	\begin{align*}
		(\phi_1^L)'(\bar{p}; ~ exp_{\bar{p}}^{-1} ~p) &= 0,\\
		(\phi_1^U)'(\bar{p}; ~ exp_{\bar{p}}^{-1} ~p) &= 0,\\
		(\phi_2^L)'(\bar{p}; ~ exp_{\bar{p}}^{-1} ~p) &= ln(p_2),\\
		(\phi_2^U)'(\bar{p}; ~ exp_{\bar{p}}^{-1} ~p) &= ln(p_2),\\
		\psi_1'(\bar{p}; ~ exp_{\bar{p}}^{-1} ~p) &= 0,\\
		\psi_2'(\bar{p}; ~ exp_{\bar{p}}^{-1} ~p) &= 0,\\
		\psi_3'(\bar{p}; ~ exp_{\bar{p}}^{-1} ~p) &= -ln(p_1),\\
		\psi_4'(\bar{p}; ~ exp_{\bar{p}}^{-1} ~p) &= 2ln(p_1) + 2ln(p_2),\\
		\psi_5'(\bar{p}; ~ exp_{\bar{p}}^{-1} ~p) &= -ln(p_2).
	\end{align*}
	It is easy to check that conditions (i), (ii) and (iii) in Theorem \ref{theorem3.4,p1}(a) hold at $\bar{p} = (1,1)$ with $(\lambda_1^L, \lambda_2^L, \mu_1^L, \mu_2^L, \mu_3^L, \mu_4^L, \mu_5^L) =(\lambda_1^U, \lambda_2^U, \mu_1^U, \mu_2^U, \mu_3^U, \mu_4^U, \mu_5^U)=(1,1,0,0,2,1,3)$. So, by Theorem \ref{theorem3.4,p1}(a), we conclude that $\bar{p} = (1,1)$ is a type-I POS of (MIVOP4).
\end{example}

\subsection{KKT Conditions for Weakly Type-I and Type-II POS}
We now address certain KKT conditions for weakly type-I POS by using the concepts of LU-convexity and LU-pseudo-convexity as well as few conditions for weakly type-II POS by using CW-covexity and  CW-pseudo-convexity for (MIVOP2).

\begin{theorem}\label{theorem3.5,p1}
	Let $\psi_u : E \rightarrow \mathbb{R}$, $u \in \{1,2,...,t\}$, be convex at $\bar{p} \in F$;
	\begin{enumerate}[label=\bf(\alph*)]
		\item If, there exists $c \in \{1,2,...,l\}$, such that $\phi_c : E \rightarrow \mathbb{I}$ with $\phi_c(p) = [\phi_c^L(p), \phi_c^U(p)]$ is LU-convex and weakly directionally differentiable at $\bar{p}$ and if, there exist, $0 < \lambda^L,~ \lambda^U \in \mathbb{R}$ and $0 \leqq \mu_u \in \mathbb{R}$, $u \in \{1,2,...,t\}$, such that:
		\begin{enumerate}[label=(\roman*)]
			\item $\lambda^L (\phi_c^L)'(\bar{p};~ exp^{-1}_{\bar{p}} ~ p) + \lambda^U (\phi_c^U)'(\bar{p};~ exp^{-1}_{\bar{p}} ~ p) + \displaystyle\sum_{u=1}^{t} \mu_u \psi_u'(\bar{p};~ exp^{-1}_{\bar{p}} ~ p) \geqq 0;$
			\item $\mu_u\psi_u(\bar{p}) = 0$, ~ $u \in \{1,2,...,t\}$,
		\end{enumerate}
	\end{enumerate}
	then, $\bar{p}$ is a weakly type-I POS of (MIVOP2).
	\begin{enumerate}[resume*]
		\item If, there exists $c \in \{1,2,...,l\}$, such that $\phi_c : E \rightarrow \mathbb{I}$ with $\phi_c(p) = [\phi_c^L(p), \phi_c^U(p)]$ is CW-convex and weakly directionally differentiable at $\bar{p}$ and if, there exist, $0 < \lambda^C,~ \lambda^W \in \mathbb{R}$ and $0 \leqq \mu_u \in \mathbb{R}$, $u \in \{1,2,...,t\}$, such that:
		\begin{enumerate}[label=(\roman*)]
			\item[(iii)] $\lambda^C (\phi_c^C)'(\bar{p};~ exp^{-1}_{\bar{p}} ~ p) + \lambda^W (\phi_c^W)'(\bar{p};~ exp^{-1}_{\bar{p}} ~ p) + \displaystyle\sum_{u=1}^{t} \mu_u \psi_u'(\bar{p};~ exp^{-1}_{\bar{p}} ~ p) \geqq 0;$
			\item[(iv)] $\mu_u\psi_u(\bar{p}) = 0$, ~ $u \in \{1,2,...,t\}$,
		\end{enumerate}
	\end{enumerate}
	then, $\bar{p}$ is a weakly type-II POS of (MIVOP2).
	\begin{enumerate}[resume*]
		\item If, there exists $c \in \{1,2,...,l\}$, such that $\phi_c : E \rightarrow \mathbb{I}$ with $\phi_c(p) = [\phi_c^L(p), \phi_c^U(p)]$ is CW-convex and weakly directionally differentiable at $\bar{p}$ and if, there exist, $0 < \lambda^L,~ \lambda^U \in \mathbb{R}$ and $0 \leqq \mu_u \in \mathbb{R}$, $u \in \{1,2,...,t\}$, such that:
		\begin{enumerate}[label=(\roman*)]
			\item[(v)] $\lambda^L < \lambda^U$;
			\item[(vi)] $\lambda^L (\phi_c^L)'(\bar{p};~ exp^{-1}_{\bar{p}} ~ p) + \lambda^U (\phi_c^U)'(\bar{p};~ exp^{-1}_{\bar{p}} ~ p) + \displaystyle\sum_{u=1}^{t} \mu_u \psi_u'(\bar{p};~ exp^{-1}_{\bar{p}} ~ p) \geqq 0;$
			\item[(vii)] $\mu_u\psi_u(\bar{p}) = 0$, ~ $u \in \{1,2,...,t\}$,
		\end{enumerate}
	\end{enumerate}
	then, $\bar{p}$ is a weakly type-II POS of (MIVOP2).
\end{theorem}
\begin{proof}
	Define,\\ 
	$$\bar{\phi}(p) = \lambda^L \phi_c^L(p) + \lambda^U \phi_c^U(p).$$
	Since, $\phi_c(p)$ is an LU-convex function at $\bar{p}$. It follows from Lemma \ref{lemma2.2,p1} and Definition \ref{definition2.8,p1} that each $\phi_s^L(p)$ and $\phi_s^U(p)$ are convex functions and directionally differentiable at $\bar{p}$. The proof is now analogous to Theorem \ref{theorem3.2,p1}.
\end{proof}

\begin{theorem}\label{theorem3.6,p1}
	Let $\bar{p} \in F$ and assume $\psi_u:E\rightarrow \mathbb{R}$, $u \in J(\bar{p}),$ of (MIVOP2) be strictly pseudo-convex at $\bar{p}$. If, there exists $c \in \{1,2,...,l\}$, such that, $\phi_c : E \rightarrow \mathbb{I}$ with $\phi_c(p) = [\phi_c^L(p), \phi_c^U(p)]$, defined on an open geodesic convex set $E$, be weakly directionally differentiable at $\bar{p}$;
	\begin{enumerate}[label=\bf(\alph*)]
		\item If $\phi_c$ is an LU-pseudo-convex function at $\bar{p}$ and for every $p\in F$, there exist, $0 \leqq \mu_u^L,~ \mu_u^U \in \mathbb{R}$, $u \in \{1,2,...,t\}$, such that:
		\begin{enumerate}[label=(\roman*)]
			\item $(\phi_c^L)'(\bar{p}; ~exp^{-1}_{\bar{p}} ~ p)$ + $\displaystyle\sum_{u=1}^{t} \mu_u^L \psi_u'(\bar{p}; ~exp^{-1}_{\bar{p}} ~ p) \geqq 0;$
			\item $(\phi_c^U)'(\bar{p}; ~exp^{-1}_{\bar{p}} ~ p)$ + $\displaystyle\sum_{u=1}^{t} \mu_u^U \psi_u'(\bar{p}; ~exp^{-1}_{\bar{p}} ~ p) \geqq 0;$
			\item $\mu_u^L\psi_u(\bar{p}) = 0 = \mu_u^U\psi_u(\bar{p})$,  $u \in \{1,2,...,t\}$,
		\end{enumerate}
	\end{enumerate}
	then, $\bar{p}$ is a weakly type-I POS of (MIVOP2).
	\begin{enumerate}[resume*]
		\item If $\phi_c$ is CW-pseudo-convex function at $\bar{p}$ and for every $p\in F$, there exist, $0 \leqq \mu_u^C,~ \mu_u^W \in \mathbb{R}$, $u \in \{1,2,...,t\}$, such that:
		\begin{enumerate}[label=(\roman*)]
			\item[(iv)] $(\phi_c^C)'(\bar{p}; ~exp^{-1}_{\bar{p}} ~ p)$ + $\displaystyle\sum_{u=1}^{t} \mu_u^C \psi_u'(\bar{p}; ~exp^{-1}_{\bar{p}} ~ p) \geqq 0;$
			\item[(v)] $(\phi_c^W)'(\bar{p}; ~exp^{-1}_{\bar{p}} ~ p)$ + $\displaystyle\sum_{u=1}^{t} \mu_u^W \psi_u'(\bar{p}; ~exp^{-1}_{\bar{p}} ~ p) \geqq 0;$
			\item[(vi)] $\mu_u^C\psi_u(\bar{p}) = 0 = \mu_u^W\psi_u(\bar{p})$,  $u \in \{1,2,...,t\}$,
		\end{enumerate}
	\end{enumerate}
	then, $\bar{p}$ is a weakly type-II POS of (MIVOP2).
\end{theorem}

\begin{proof}
	Let, on contrary, $\bar{p}$ be not a weakly type-I POS of (MIVOP2), then, there exists $\hat{p} \neq \bar{p} \in F$, such that $\phi_s(\hat{p}) <_{LU} \phi_s(\bar{p})$, $s \in \{1,2,...,l\}$. So, we have $\phi_c(\hat{p}) <_{LU} \phi_c(\bar{p})$, which means that $\phi_c^L(\hat{p}) < \phi_c^L(\bar{p})$ or $\phi_c^U(\hat{p}) < \phi_c^U(\bar{p})$.
	
	Since, $\phi_c$ is LU-pseudo-convex function at $\bar{p}$, we have from Definition \ref{definition2.12,p1}(i) that $\phi_c^L$ and $\phi_c^U$ are pseudo-convex functions at $\bar{p}$. From Definition \ref{definition2.11,p1}(i), we have
	$$(\phi_c^L)'(\bar{p}; ~exp^{-1}_{\bar{p}} ~ \hat{p}) < 0, ~~~ \text{or} ~~~ (\phi_c^U)'(\bar{p}; ~exp^{-1}_{\bar{p}} ~ \hat{p}) < 0.$$
	The proof is now analogous to the Theorem \ref{theorem3.4,p1}.
\end{proof}

\subsection{KKT Conditions for Strongly Type-I and Type-II POS}
In this section, we address KKT conditions for strongly type-I and type-II POS of (MIVOP2).

\begin{definition}\label{definition3.1,p1}
	\rm	Suppose that an interval-valued function $\phi : E \rightarrow \mathbb{I}$ with $\phi(p) = [\phi^L(p), \phi^U(p)]$, defined on a nonempty open geodesic convex set $E \subseteq H$ is weakly directionally differentiable at $\bar{p} \in E$. We say that $\phi$ is L-pseudo-convex, (strictly L-pseudo-convex), at $\bar{p} \in E$ if $\phi^L$ is pseudo-convex function, (strictly pseudo-convex function), at $\bar{p}$.
\end{definition}

We remark here that one can similarly define the other concepts of U-pseudo-convexity, C-pseudo-convexity and W-pseudo-convexity involving the respective functions $\phi^U$, $\phi^C$ and $\phi^W$.

\begin{remark}
	\rm It is easy to see that $\phi$ is LU-pseudo-convex function, (CW-pseudo-convex function), at $\bar{p}$ if and only if $\phi$ is L-pseudo-convex function and U-pseudo-convex function, (C-pseudo-convex function and W-pseudo-convex function), at $\bar{p}$.
\end{remark}

\begin{theorem}\label{theorem3.7,p1}
	Let $\bar{p} \in F$ and assume $\psi_u:E\rightarrow \mathbb{R}$, $u \in J(\bar{p})$, of (MIVOP2) be strictly pseudo-convex at $\bar{p}$. If, there exists $c \in \{1,2,...,l\}$, such that, $\phi_c : E \rightarrow \mathbb{I}$ with $\phi_c(p) = [\phi_c^L(p), \phi_c^U(p)]$, defined on an open geodesic convex set $E$, be weakly directionally differentiable at $\bar{p}$;
	\begin{enumerate}[label=\bf(\alph*)]
		\item If $\phi_c$ is strictly L-pseudo-convex, (strictly U-pseudo-convex), at $\bar{p}$ and for every $p\in F$, there exist, $0 \leqq \mu_u \in \mathbb{R}$, $u \in \{1,2,...,t\}$, such that:
		\begin{enumerate}[label=(\roman*)]
			\item $(\phi_c^L)'(\bar{p}; ~exp^{-1}_{\bar{p}} ~ p)$ + $\displaystyle\sum_{u=1}^{t} \mu_u \psi_u'(\bar{p}; ~exp^{-1}_{\bar{p}} ~ p) \geqq 0, $ \\
			$\big((\phi_c^U)'(\bar{p}; ~exp^{-1}_{\bar{p}} ~ p) + \displaystyle\sum_{u=1}^{t} \mu_u \psi_u'(\bar{p}; ~exp^{-1}_{\bar{p}} ~ p) \geqq 0 \big);$\\
			\item $\mu_u\psi_u(\bar{p}) = 0,$ $u \in \{1,2,...,t\}$,
		\end{enumerate}
	\end{enumerate}
	then, $\bar{p}$ is a strongly type-I POS of (MIVOP2).
	\begin{enumerate}[resume*]
		\item If $\phi_c$ is strictly C-pseudo-convex, (strictly W-pseudo-convex), at $\bar{p}$ and, for every $p \in F$, there exist, $0 \leqq \mu_u\in \mathbb{R}$, $u \in \{1,2,...,t\}$, such that:
		\begin{enumerate}[label=(\roman*)]
			\item[(iii)] $(\phi_c^C)'(\bar{p}; ~exp^{-1}_{\bar{p}} ~ p)$ + $\displaystyle\sum_{u=1}^{t} \mu_u \psi_u'(\bar{p}; ~exp^{-1}_{\bar{p}} ~ p) \geqq 0, $\\
			$\big((\phi_c^W)'(\bar{p}; ~exp^{-1}_{\bar{p}} ~ p) + \displaystyle\sum_{u=1}^{t} \mu_u \psi_u'(\bar{p}; ~exp^{-1}_{\bar{p}} ~ p) \geqq 0 \big);$
			\item[(iv)] $\mu_u\psi_u(\bar{p}) = 0$, $u \in \{1,2,...,t\}$,
		\end{enumerate}
	\end{enumerate}
	then, $\bar{p}$ is a strongly type-II POS of (MIVOP2).
\end{theorem}

\begin{proof}
	On contrary, let $\bar{p}$ be not a strongly type-I POS of (MIVOP2), then by Definition \ref{definition2.5,p1}(ii), there exists $\hat{p} \neq \bar{p} \in F$, such that, $\phi(\hat{p}) \leqq_{LU} \phi(\bar{p})$, i.e., $\phi_s(\hat{p}) \leqq_{LU} \phi_s(\bar{p})$,  $s \in \{1,2,...,l\}$. Thus, we have $\phi_c(\hat{p}) \leqq_{LU} \phi_c(\bar{p})$ which gives $\phi_c^L(\hat{p}) \leqq \phi_c^L(\bar{p})$, ($\phi_c^U(\hat{p}) \leqq \phi_c^U(\bar{p})$). Since $\phi_c$ is strictly L-pseudo-convex function, (strictly U-pseudo-convex function), at $\bar{p}$ which means that $\phi_c^L$, ($\phi_c^U$), is strictly pseudo-convex function at $\bar{p}$. By Definition \ref{definition2.11,p1}(ii), we now have
	$$(\phi_c^L)'(\bar{p}; ~ exp^{-1}_{\bar{p}}~\hat{p}) < 0, ~~ ((\phi_c^U)'(\bar{p}; ~ exp^{-1}_{\bar{p}}~\hat{p}) < 0).$$
	The proof is now analogous to Theorem \ref{theorem3.6,p1}.
\end{proof}

\begin{definition}\label{definition3.2,p1}
	\rm A function $\phi : E \rightarrow \mathbb{I}$ with $\phi(p) = [\phi^L(p), \phi^U(p)]$, defined on a nonempty geodesic convex set $E \subseteq H$, is L-convex at $\bar{p} \in E$ if and only if $\phi^L$ is a convex function at $\bar{p}$.
\end{definition}

We remark here that one can similarly define the other concepts of U-convexity, C-convexity and W-convexity involving the respective functions $\phi^U$, $\phi^C$ and $\phi^W$.

\begin{remark}
	\rm It is easy to see that $\phi$ is LU-convex function, (CW-convex function), at $\bar{p}$ if and only if $\phi$ is L-convex function and U-convex function, (C-convex function and W-convex function), at $\bar{p}$.
\end{remark}

\begin{theorem}\label{theorem3.8,p1}
	Let $\bar{p} \in F$ and assume $\psi_u:E\rightarrow \mathbb{R}$, $u \in J(\bar{p})$, of (MIVOP2) be strictly pseudo-convex at $\bar{p}$. If, there exists $c \in \{1,2,...,l\}$, such that, $\phi_c : E \rightarrow \mathbb{I}$ with $\phi_c(p) = [\phi_c^L(p), \phi_c^U(p)]$, defined on an open geodesic convex set $E$, be weakly directionally differentiable at $\bar{p}$ such that
	\begin{enumerate}[label=(\roman*)]
		\item $(\phi_c^L)'(\bar{p}; ~exp^{-1}_{\bar{p}} ~ p) < (\phi_c^U)'(\bar{p}; ~exp^{-1}_{\bar{p}} ~ p)$;\\
		
		$\big((\phi_c^U)'(\bar{p}; ~exp^{-1}_{\bar{p}} ~ p)< (\phi_c^L)'(\bar{p}; ~exp^{-1}_{\bar{p}} ~ p)\big),$
	\end{enumerate}
	then the following arguments hold true;
	\begin{enumerate}[label=\bf (\alph*)]
		\item If $\phi_c$ is U-convex function, (L-convex function), at $\bar{p}$ and for every $p\in F$, there exist, $0 \leqq \mu_u\in \mathbb{R}$, $u \in \{1,2,...,t\}$, such that:
		\begin{enumerate}[label=(\roman*)]    
			\item[(ii)] $(\phi_c^L)'(\bar{p}; ~exp^{-1}_{\bar{p}} ~ p)$ + $\displaystyle\sum_{u=1}^{t} \mu_u \psi_u'(\bar{p}; ~exp^{-1}_{\bar{p}} ~ p) \geqq 0, $\\ $\big((\phi_c^U)'(\bar{p}; ~exp^{-1}_{\bar{p}} ~ p) + \displaystyle\sum_{u=1}^{t} \mu_u \psi_u'(\bar{p}; ~exp^{-1}_{\bar{p}} ~ p) \geqq 0 \big);$
			\item[(iii)] $\mu_u\psi_u(\bar{p}) = 0,$ $u \in \{1,2,...,t\}$,
		\end{enumerate}
	\end{enumerate}
	then, $\bar{p}$ is a strongly type-I POS of (MIVOP2).
	\begin{enumerate}[resume*]
		\item If $\phi_c$ is C-convex function, (W-convex function), at $\bar{p}$ and, for every $p\in F$, there exist, $0 \leqq \mu_u\in \mathbb{R}$, $u \in \{1,2,...,t\}$, such that:
		\begin{enumerate}[label=(\roman*)]  
			\item[(iv)] $(\phi_c^L)'(\bar{p}; ~exp^{-1}_{\bar{p}} ~ p) >0,$ $\big((\phi_c^U)'(\bar{p}; ~exp^{-1}_{\bar{p}} ~ p) < 0\big);$
			\item[(v)] $(\phi_c^C)'(\bar{p}; ~exp^{-1}_{\bar{p}} ~ p) + \displaystyle\sum_{u=1}^{t} \mu_u \psi_u'(\bar{p}; ~exp^{-1}_{\bar{p}} ~ p) \geqq 0,$\\
			$\big((\phi_c^W)'(\bar{p}; ~exp^{-1}_{\bar{p}} ~ p) + \displaystyle\sum_{u=1}^{t} \mu_u \psi_u'(\bar{p}; ~exp^{-1}_{\bar{p}} ~ p) \geqq 0 \big);$
			\item[(vi)] $\mu_u\psi_u(\bar{p}) = 0$, $u \in \{1,2,...,t\}$,
		\end{enumerate}
	\end{enumerate}
	then, $\bar{p}$ is a strongly type-II POS of (MIVOP2).	
\end{theorem}
\begin{proof}
	Let, on contrary, $\bar{p}$ be not a strongly type-I POS of (MIVOP2), then, there exists $\hat{p} \neq \bar{p} \in F$, such that, $\phi(\hat{p}) \leqq_{LU} \phi(\bar{p})$, i.e., $\phi_s(\hat{p}) \leqq_{LU} \phi_s(\bar{p})$, $s=1,2,...,l$. So, $\phi_c(\hat{p}) \leqq_{LU} \phi_c(\bar{p})$, which implies $\phi_c^U(\hat{p}) \leqq_{LU} \phi_c^U(\bar{p})$. Since, $\phi_c(p)$ is U-convex function, from Definition \ref{definition3.2,p1}, $\phi_c^U(p)$ is convex function. Hence, we have from Remark \ref{remark2.1,p1}(ii), that,\\
	\begin{equation}\label{equation14,p1}
		(\phi_c^U)'(\bar{p};~ exp^{-1}_{\bar{p}} ~ \hat{p}) \leqq 0.
	\end{equation}
	Also, we are given that\\
	\begin{equation}\label{equation15,p1}
		(\phi_c^L)'(\bar{p};~ exp^{-1}_{\bar{p}} ~ \hat{p}) < (\phi_c^U)'(\bar{p};~ exp^{-1}_{\bar{p}} ~ \hat{p}).
	\end{equation}
	From inequations (\ref{equation14,p1}) and (\ref{equation15,p1}), we have \\
	$$(\phi_c^L)'(\bar{p};~ exp^{-1}_{\bar{p}} ~ \hat{p}) <0.$$
	Part {\bf (a)} can now be obtained from the similar arguments of theorem \ref{theorem3.6,p1}.\\
	For the part {\bf(b)}, we first note that, for an interval $A=[a^L,a^U]$ with $0<a^L<a^U$, we have\\
	$$\frac{a^U-a^L}{2}<\frac{a^L+a^U}{2}$$
	$$\text{i.e.,}  ~~~~~~ a^W<a^C~~~~~~~~~~~$$
	As a consequence, it is straight forward to see, from conditions (i) and (iv), that\\
	$~~~~~~~~~~~~~~~~~~~~(\phi_c^W)'(\bar{p}; ~exp^{-1}_{\bar{p}} ~ p) < (\phi_c^C)'(\bar{p}; ~exp^{-1}_{\bar{p}} ~ p),$\\ $~~~~~~~~~~\big(\text{respectively,~} (\phi_c^C)'(\bar{p}; ~exp^{-1}_{\bar{p}} ~ p) < (\phi_c^W)'(\bar{p}; ~exp^{-1}_{\bar{p}} ~ p)\big).$\\
	The result can now be obtained similarly as that of {\bf(a)}.     
\end{proof}

\section{KKT Conditions for Generalized Hukuhara Difference Case}\label{section4,p1}
In this section, we will address KKT conditions using gH-difference. 

\begin{theorem}\label{theorem4.1,p1}
	Let $\psi_u: E \rightarrow \mathbb{R}$, $u\in \{1,2,...,t\}$, be convex at $\bar{p} \in F$ and let $\phi_s: E \rightarrow \mathbb{I}$ with $\phi_s(p) = [\phi_s^L(p), \phi_s^U(p)]$, $s\in \{1,2,...,l\}$, be LU-convex and weakly directionally differentiable at $\bar{p} \in E$. Suppose, there exist, $0 < \lambda_s \in \mathbb{R}$, $s\in \{1,2,...,l\}$, and $0 \leqq \mu_u \in \mathbb{R}$, $u \in \{1,2,...,t\}$, such that:
	\begin{enumerate}[label=(\roman*)]
		\item -$\big($$\displaystyle \sum_{s=1}^{l} \lambda_s \phi'_s(\bar{p};~exp^{-1}_{\bar{p}} ~p)\big) \ominus_g \displaystyle\sum_{u=1}^{t} \mu_u \psi_u'(\bar{p};~exp^{-1}_{\bar{p}} ~p) \leqq_{LU} [0,0];$
		\item $\mu_u\psi_u(\bar{p}) = 0$, $u \in \{1,2,...,t\}$,
	\end{enumerate}
	then, $\bar{p}$ is a type-I POS of (MIVOP2).
\end{theorem} 

\begin{proof}
	The condition (i) yields
	\begin{equation}\label{equation16,p1}
		\displaystyle \sum_{s=1}^{l} \lambda_s \text{min}\{(\phi^L_s)'(\bar{p};~exp^{-1}_{\bar{p}} ~p),~(\phi^U_s)'(\bar{p};~exp^{-1}_{\bar{p}} ~p)\} + \displaystyle\sum_{u=1}^{t} \mu_u \psi_u'(\bar{p};~exp^{-1}_{\bar{p}} ~p) \geqq 0 \hspace*{0cm}
	\end{equation}
	\begin{equation}\label{equation17,p1}
		\displaystyle \sum_{s=1}^{l} \lambda_s \text{max}\{(\phi^L_s)'(\bar{p};~exp^{-1}_{\bar{p}} ~p),~(\phi^U_s)'(\bar{p};~exp^{-1}_{\bar{p}} ~p)\} + \displaystyle\sum_{u=1}^{t} \mu_u \psi_u'(\bar{p};~exp^{-1}_{\bar{p}} ~p) \geqq 0. \hspace*{0cm}
	\end{equation}
	Adding inequalities (\ref{equation16,p1}) and (\ref{equation17,p1}), we have
	\begin{equation}\label{equation18,p1}
		\begin{split}
			\displaystyle \sum_{s=1}^{l} \lambda_s (\phi^L_s)'(\bar{p};~exp^{-1}_{\bar{p}} ~p) &+ \displaystyle \sum_{s=1}^{l} \lambda_s (\phi^U_s)'(\bar{p};~exp^{-1}_{\bar{p}} ~p)\\ 
			&+ \displaystyle\sum_{u=1}^{t} \mu_u' \psi_u'(\bar{p};~exp^{-1}_{\bar{p}} ~p) \geqq 0,~~~~~~~~~~~~
		\end{split}
	\end{equation}
	$\text{where~~~~}\mu_u' = 2\mu_u$, ~~ $u \in \{1,2,...,t\}.$\\ 
	Now, letting $\lambda_s = \lambda_s^L = \lambda_s^U$ and assuming $\mu_u = \mu_u'$ in Theorem \ref{theorem3.2,p1}(a), the result follows immediately.
\end{proof}

\begin{remark}
	\rm	The KKT conditions for gH-difference of the remaining results described in prequel are analogous to Theorem \ref{theorem4.1,p1}.
\end{remark}

\begin{example}
	{\rm \bfseries Manifold (or Cone) of symmetric positive definite matrices: }
	\rm The collection, $S^n_{++}$, of $n \times n$ symmetric positive definite matrices with real entries forms a Hadamard manifold with Riemannian metric:
	$$g_P(X,Y) = Tr(P^{-1}XP^{-1}Y), ~~~ \forall~ P \in S^n_{++}, ~~~ X,Y \in T_P(S^n_{++}).$$
	The minimal geodesic joining $P, Q \in S^n_{++}$ is given by 
	$$\gamma(t) = P^{\frac{1}{2}}(P^{-\frac{1}{2}}QP^{-\frac{1}{2}})^tP^{\frac{1}{2}}, ~~~ \forall ~ t \in [0,1].$$
	For more details, one can refer to \cite{bacak}.
	
	We consider $S^2_{++}$ as the Hadamard manifold	and let $P=I_2$ be a $2 \times 2$ identity matrix and $Q \in S^2_{++}$ be any arbitrary $2 \times 2$ matrix. Then the geodesic emanating from $P=I_2$ in the direction $W = exp^{-1}_P~Q \in T_P(S_{++}^2)$ is given by
	$$\gamma(t) = I_2^{\frac{1}{2}}(I_2^{-\frac{1}{2}}QI_2^{-\frac{1}{2}})^tI_2^{\frac{1}{2}} = Q^t,~~ \text{for all}~ t \in [0,1].$$
	Let $\phi_s: S^2_{++} \rightarrow \mathbb{I}$, $s \in \{1,2\}$, be interval-valued functions and $\psi_u: S^2_{++} \rightarrow \mathbb{R}$, $u \in \{1,2,3\}$, be the real valued constraint functions for the following.
	\begin{align*}
		\text{(MIVOP5)\hspace{0.5cm}Minimize~~~~~} \phi(Q) &= (\phi_1(Q),~ \phi_2(Q)),\\
		\text{where }~~~ \phi_1(Q) &= \big[ln(det(Q)),~ ln(det(Q)) +1\big],\\
		\text{and}~~~~ \phi_2(Q) &= \big[(ln(det(Q)))^2,~ (ln(det(Q)))^2 +1\big],\\
		\text{subject to,~~~} \psi_1(Q) &= \frac{1}{(1 + ln(det(Q)))^2} - 1 \leqq 0,\\
		\psi_2(Q) &= -ln(det(Q)) - 3\leqq 0,\\
		\psi_3(Q)&= \sqrt{3 + (ln(det(Q)))^2} - \sqrt{3}\leqq 0.
	\end{align*}
	Here, both the objective functions $\phi_1(Q)$ and $\phi_2(Q)$ are LU-convex at $P=I_2$ and the constraints $\psi_u(Q)$, $u \in \{1,2,3\}$ are also convex functions at $P=I_2$.
	
	The gH-directional derivatives of $\phi_1(Q)$ and $\phi_2(Q)$ at $P=I_2$ in any direction $w= exp^{-1}_P ~Q \in T_P(S^2_{++})$ are as follows
	\begin{align*}
		(\phi_1)'(P; ~ exp_P^{-1} ~Q) &=\big[ln(det(Q)),~ ln(det(Q))\big] ,\\
		(\phi_2)'(P; ~ exp_P^{-1} ~Q) &= [0,0].
	\end{align*}
	The directional derivatives of the constraint functions involved in the optimization problem at $P=I_2$ in any direction $w= exp^{-1}_P ~Q \in T_P(S^2_{++})$, are given as follows
	\begin{align*}
		\psi_1'(P; ~ exp_P^{-1} ~Q) &= -2(ln(det(Q))),\\
		\psi_2'(P; ~ exp_P^{-1} ~Q) &= -ln(det(Q)),\\
		\psi_3'(P; ~ exp_P^{-1} ~Q) &= 0.
	\end{align*}
	It is easy to check that conditions (i) and (ii) in Theorem \ref{theorem4.1,p1} hold at $P=I_2$ with $(\lambda_1, \lambda_2, \mu_1, \mu_2, \mu_3)= (2,1,1,0,1)$. So, by \ref{theorem4.1,p1}, we conclude that $P=I_2$ is a type-I POS of (MIVOP5).
\end{example}

\section{Comparison between proposed KKT conditions and the existing KKT conditions for MIVOP}
In this section, we make comparison between the proposed KKT optimality conditions and the existing ones for interval-valued optimization problems. The comparison is based on the following facts.
\begin{enumerate}[label=(\roman*)]
	\item There exist interval-valued functions which fail to be LU-convex in an Euclidean space but turn out to be LU-convex in a Riemannian manifold under a suitable Riemannian metric, see Example \ref{example3.1,p1}. Hence, the KKT optimality conditions presented in (\cite{ycc}, \cite{eh},\cite{Wu}, \cite{Wu1}, \cite{jz}) do not apply to the optimization programming problems which involve such functions.
	
	\item The KKT optimality conditions presented in (\cite{eh, Wu, Wu1, jz}) use the idea of H-difference of two intervals and H-derivative of an interval-valued function. It is well known that the H-difference of two intervals doesn't always exist but gH-difference always exists. Also, there exist interval-valued functions for which gH-directional derivative exists but may fail to have H-derivative. For instance, the function $\phi: \mathbb{R} \rightarrow \mathbb{I}$ defined by $\phi(p) = (1-p^3)[-1, 4]$ fails to have H-derivative at $P=0$ but the gH-directional derivative of $\phi(p)$ at $p=0$ in the direction $w=1$ is $\phi'(p;~w)=[0,0]$. However, the proposed KKT conditions presented in Section \ref{section4,p1} use the idea of gH-directional derivative which is more general than the H-derivative.
	
	\item The KKT condition used in (\cite{ycc}, \cite{eh},\cite{Wu}, \cite{Wu1}, \cite{jz}) is of the form
	$$\displaystyle \sum_{s=1}^{l} \lambda_s^L \nabla\phi_s^L(\bar{p}) + \displaystyle \sum_{s=1}^{l} \lambda_s^U \nabla\phi_s^U(\bar{p})
	+ \displaystyle\sum_{u=1}^{t} \mu_u \nabla\psi_u(\bar{p}) = 0,$$
	which is very restrictive as the expression in Left hand side involves gradient of objective functions and constraints. Also, the expression in the left hand side needs to be exactly zero. However, the proposed KKT condition in this paper is very mild which involves the one sided directional derivative of the objective functions and constraints. Also, the proposed KKT condition is of the form
	\begin{align*}
		\displaystyle \sum_{s=1}^{l} \lambda_s^L (\phi_s^L)'(\bar{p};~ exp^{-1}_{\bar{p}} ~ p) &+ \displaystyle \sum_{s=1}^{l} \lambda_s^U (\phi_s^U)'(\bar{p};~ exp^{-1}_{\bar{p}} ~ p)\\ 
		&+ \displaystyle\sum_{u=1}^{t} \mu_u \psi_u'(\bar{p};~ exp^{-1}_{\bar{p}} ~ p) \geqq 0,
	\end{align*}  
	which is less restrictive as the expression in the left hand side no longer needs to be exactly zero.
\end{enumerate}

\section{Relation to Fuzzy Optimization Programming}
The relationship between fuzzy set theory and interval analysis is well known. Lodwick \cite{lod1} studied the interconnections between interval analysis, fuzzy interval analysis, and interval and fuzzy optimization. The fuzzy interval analysis is connected with the real arithmetics via constraint interval arithmetics on alpha-cuts and arithmetics on intervals of gradual numbers. Lodwick \cite{lod2} studies the connection between the two methods. One can extend the important results developed in interval analysis into the fuzzy interval analysis using these two tools.

Moreover, the ordering relation, solution concepts and the results presented herein under gH-differentiability can be extended to the space of fuzzy optimization as done by Wu\cite{wu3}. Although, gH-difference between two intervals always exists but the same may fail to exist between two fuzzy intervals. Stefananini \cite{stef4} proposed g-difference to overcome this shortcoming for fuzzy intervals. As g-difference between two fuzzy intervals always exists, one can use this concept to extend the results presented in this paper into fuzzy intervals and fuzzy interval-valued optimization programming. Recently, Huidobro\cite{hui5} have extended the concepts of convexity for interval-valued fuzzy sets and studied the preservation under the intersection and the cutworthy property, it is future work to generalize these notions of convexity to obtain the KKT optimality conditions, as presented in this paper, into the space of interval-valued fuzzy optimization programming. In this sense, this research is a step forward to obtain KKT optimality coditions for multi-objective interval-valued fuzzy optimization programming.

\section{Conclusion}
On Hadamard manifolds, we have successfully obtained the Karush-Kuhn-Tucker optimality conditions for multi-objective interval-valued optimization problems, and we have offered numerical examples to demonstrate the major conclusions. It would be interesting to foresee the idea of using these results to generate algorithms for obtaining Pareto optimal points, as well as to examine certain practical issues such as portfolio selection problem, hub network designing, supplier selection and many more. These results could also be extended to more general spaces, such as Riemannian manifolds. However, because the constraint functions involved in investigating the Pareto optimal solution of multi-objective interval-valued optimization problems are real valued, these ideas could be extended in the future by taking into account interval-valued constraint functions.


\end{document}